\documentclass[11pt]{article}

\usepackage{tikz,amsthm,amsmath,amstext,amssymb,epsfig,euscript, mathrsfs, dsfont, enumerate,multicol,graphics,graphicx,times}
\usepackage{hyperref}

\def\ve{\varepsilon}

\def\sothat{\mbox{ s.t. }}

\def\bC{{\mathbb{C}}}
\def\bD{{\mathbb{D}}}
\def\bR{{\mathbb{R}}}

\renewcommand{\d}{{\partial}}

\def\cB{{\mathscr{B}}}
\def\cC{{\mathscr{C}}}

\def\cH{{\mathscr{H}}}
\def\cI{{\mathscr{I}}}

\def\cN{{\mathscr{N}}}

\def\cQ{{\mathscr{Q}}}

\def\star{({\bf C1})\ }

\newcommand{\ps}[1]{\left( #1 \right)}

\newcommand{\cnj}[1]{\overline{#1}}
\newcommand{\isif}[1]{\left\{\begin{array}{cc} #1
\end{array}\right.}

\def\dist{\mbox{dist}}

\def\otimes{\otimes\cdots\otimes}

\def\lec{\lesssim}
\def\gec{\gtrsim}

\def\barintgerm_#1{\mathchoice
{\mathop{\vrule width 6pt height 3 pt depth -2.5pt
\kern -8.8pt \intop}\nolimits_{#1}}%
{\mathop{\vrule width 5pt height 3 pt depth -2.6pt
\kern -6.5pt \intop}\nolimits_{#1}}%
{\mathop{\vrule width 5pt height 3 pt depth -2.6pt
\kern -6pt \intop}\nolimits_{#1}}%
{\mathop{\vrule width 5pt height 3 pt depth -2.6pt \kern -6pt
\intop}\nolimits_{#1}}}

\def\diam{\mbox{diam}}

\newcommand{\bc}{\begin{center}}
\newcommand{\ec}{\end{center}}
\newcommand{\be}{\begin{equation}}
\newcommand{\ee}{\end{equation}}
\newcommand{\bd}{\begin{description}}
\newcommand{\ed}{\end{description}}

\newcommand{\ei}{\end{itemize}}
\newcommand{\bn}{\begin{enumerate}}
\newcommand{\en}{\end{enumerate}}
\newcommand{\ba}{\begin{array}}
\newcommand{\ea}{\end{array}}

\newcommand{\bt}{\begin{tabular}}
\newcommand{\et}{\end{tabular}}

\newcounter{prop}
\theoremstyle{plain}
\newtheorem{theorem}{Theorem}[section]

\newtheorem{lemma}[theorem]{Lemma}
\newtheorem{conjecture}[theorem]{Conjecture}
\newtheorem{proposition}[theorem]{Proposition}
\newtheorem{corollary}[theorem]{Corollary}
\newtheorem{remark}[theorem]{Remark}

\numberwithin{equation}{section}

\def\Claim{ {\bf Claim: }}

 \newcommand{\pic}[4]{
 \begin{figure*}[h]
 \begin{center}
{\includegraphics[width=#1]{#2}}
 \caption{#3}\label{fig:#4}
\end{center}
\end{figure*}}

\newcommand{\typezero}{Flat-Good\,}
\newcommand{\typeone}{Flat-Bad\,}
\newcommand{\typetwo}{Non-Flat\,}

\begin{document}

\pagestyle{myheadings}

\title{How to take shortcuts in Euclidean space: making a given set into  a short quasi-convex set}
\author{
Jonas Azzam\footnote{
	Dep. of Mathematics,
	Univ. of Washington,
	Box 354350,
	Seattle, WA 98195-4350,
	U.S.A.
} \ and Raanan Schul\footnote{
	Dep. of Mathematics,
	Stony Brook Univ.,
	Stony Brook, NY 11794-3651,
	U.S.A.}}
\maketitle

\begin{abstract}
For a given connected set $\Gamma$ in $d-$dimensional Euclidean space,
we construct a connected set $\tilde\Gamma\supset \Gamma$ such that the two sets have comparable Hausdorff length, and the set  $\tilde\Gamma$ has the property that it is quasiconvex, i.e.
any two points $x$ and $y$ in $\tilde\Gamma$ can be connected via a path, all of which is in $\tilde\Gamma$, which has length bounded by a fixed constant multiple of the Euclidean distance between $x$ and $y$.
Thus, for any set $K$ in $d-$dimensional Euclidean space we have a set $\tilde\Gamma$ as above such that $\tilde\Gamma$ has comparable Hausdorff length
to a shortest connected set containing $K$. 
Constants appearing here depend only on the ambient dimension $d$.
In the case where  $\Gamma$  is Reifenberg flat, our constants are also independent the dimension $d$, and in this case,  our theorem holds for $\Gamma$ in an infinite dimensional Hilbert space.
This work  closely related to  $k-$spanners, which appear in computer science.

\end{abstract}

\tableofcontents


\section{Statement of main theorem}
For a 
curve $\gamma$ in $\bR^d$, let $\ell(\gamma)$ denote the arclength of $\gamma$.
For a set $K\subset \bR^d$, 
let  $\cH^1(K)$ denote the $1-$dimensional Hausdorff measure of $K$.
We prove the following theorem.
\begin{theorem}\label{t:main-theorem}
Let $d\geq 2$. There exist constants $C_1,C_2>1$, $C_{1}$ depending on $d$, such that for any subset $K\subset \bR^{d}$ there exists a connected set $\tilde\Gamma\subset\bR^d$ such that:
\begin{enumerate}
\item[(i)] 
$\tilde\Gamma\supset K$.
\item[(ii)] 
$\cH^{1}(\tilde\Gamma)\leq C_1 \cH^{1}(\Gamma)$ for any  connected $\Gamma\supset K$.
\item[(iii)] 
For any $x,y\in \tilde\Gamma$ there is a path connecting $x$ and $y$, $\gamma_{x,y}\subset \tilde\Gamma$, with 
$$\ell(\gamma_{x,y})\leq C_2|x-y|\,.$$ 
\end{enumerate}
\end{theorem}
\noindent
A set $\tilde\Gamma$ satisfying property (iii) above is called quasiconvex.

The case $d=2$ was first shown by Peter Jones \cite{Jones-TSP} using complex analysis machinery. This was a main tool in his proof of the planar analyst's Traveling Salesman Theorem.  

Let us mention a relation to computer-science.
For a (possibly weighted) graph $G=(V,E)$, a k-spanner is a subgraph with the same vertices, $G'=(V,E')$, in which every two vertices are at most $k$ times as far apart on $G'$ (in the graph metric) than on $G$. 
This is a useful concept in studying  network optimization.
A geometric $k-$spanner is a graph over a set of vertices $K$ in Euclidean space, such that the graph distance is bounded by $k$ times the Euclidean distance for 
any two points in $K$. See \cite{mitchell-spn-04,spanner-book} for more details on how these are useful  in computer science.
We note that the problem we are dealing with is harder than finding k-spanners.  
For a given set $K$, we are concerned with finding a `not too long' set $\tilde\Gamma$, such that $\tilde\Gamma$ is a geometric $k-$spanner for itself, not just for the set $K$, in particular we are building a network which is not too long, and in which all new nodes are also well connected. 
(Also note that in our case, we must also treat the edges as continua of nodes.)

The Traveling Salesman Theorem is a major tool used in our proof (see Theorem \ref{thm:TST} below). It holds in the setting of an infinite dimensional Hilbert space. This is one reason why the  authors  believe the following.
\begin{conjecture}
Theorem \ref{t:main-theorem} holds with constants independent of dimension and in fact holds in the for the case where $K$ is a subset of an infinite dimensional Hilbert space.
\end{conjecture}
See Remark \ref{thm:wesuck}  for a discussion of where our present proof breaks down in this context.
Under some flatness assumptions, however, we can say more.  A set $K$ is called $\ve-$Reifenberg flat (with holes) if for any ball $B$ of radius $r$, we have that  $K\cap B$ is contained inside a tube of radius $\ve r$, where $\ve>0$ is some fixed constant.

In Remark \ref{remark:reifenberg-proof} 
 below, we indicate how our proof of Theorem \ref{t:main-theorem}, coupled with the proof of Theorem \ref{thm:TST}, yields the following theorem.
(We, unfortunately, must appeal to the proof of Theorem \ref{thm:TST}, and not its statement.)

\begin{theorem}\label{t:revisit-reifenberg}
There exist constants $C_1,C_2>1$, and $\ve>0$ such that  for any $\ve$-Reifenberg flat (with holes) set  $K\subset \cH$, a (possibly infinite-dimensional) Hilbert space,  there exists a connected set $\tilde\Gamma\subset\cH$ such that:
\begin{enumerate}
\item[(i)] 
$\tilde\Gamma\supset K$.
\item[(ii)] 
$\cH^{1}(\tilde\Gamma)\leq C_1 \cH^{1}(\Gamma)$ for any  connected $\Gamma\supset K$.
\item[(iii)] 
For any $x,y\in \tilde\Gamma$ there is a path connecting $x$ and $y$, $\gamma_{x,y}\subset \tilde\Gamma$, with 
$$\ell(\gamma_{x,y})\leq C_2|x-y|\,.$$ 
\end{enumerate}
\end{theorem}


We note that the work presented in this paper is not the first extension of the $d=2$ version of Theorem  \ref{t:main-theorem}.  
The following theorem holds for any Euclidean  space.

\begin{theorem}[\cite{Jones-TSP,GJM}]
There is $M>0$ such that if $\Gamma$ is a rectifiable simple closed curve in $\bR^{n}$ and $F(\bD)$ is a minimal surface with boundary $\Gamma$, then there is a locally finite partition $\{D_{j}\}$ of $\bD$  such that: 
\begin{enumerate}
\item $F$ is a homeomorphism of $\cnj{D}_{j}$ onto $\cnj{F(D_{j})}$,
\item $F(\d(D_{j}))$ is an $M$ chord-arc curve, and
\item $\sum \cH^{1}(F(\d(D_{n})))\leq M\cH^{1}(\Gamma)$.
\end{enumerate}
\label{thm:AGJM}
\end{theorem}

As discussed before, for $d=2$ this was done in \cite{Jones-TSP}. For $d>2$ this was shown by  John Garnett, Peter Jones and Donald Marshall \cite{GJM}.  
They adapted  the analytic techniques of Jones' original argument to the minimal surface spanned by $\Gamma$.

Other related works are, for example,  
\cite{KK} and \cite{Das-Narasimhan}. Kenyon and Kenyon \cite{KK} is  a mathematically weaker version of Theorem \ref{t:main-theorem} for $d=2$, which has the advantage that is computationally tractable. 
Das and Narasimhan,  in  \cite{Das-Narasimhan},  improve on \cite{KK} and extend to  $d>2$.
Both of these fit within the k-spanner setting in that they are concerned only with  the well-connectedness of nodes in the original set $K$, and not with the well-connectedness of the resulting set. Christopher Bishop \cite{Bishop-Tree}, improves on Theorem \ref{t:main-theorem} for $d=2$. This work  has the advantage of being computationally tractable.

\subsection{Organization}
The paper is organized as follows. 
In Section \ref{s:notation-and-tools} we set-up some notation and tools we will use. In particular we denote by $\Gamma$ a connected set of shortest Hausdorff length containing $K$.
In Section \ref{s:constructing-shortcuts} we add the needed paths to $\Gamma$, giving us a connected set $\tilde \Gamma$ which does not have length more than a constant times that of $\Gamma$.
In Section \ref{s:route_finding} we show that $\tilde\Gamma$  satisfies the properties of Theorem \ref{t:main-theorem}, in particular, that any points $x$ and $y$ in $\tilde\Gamma$ can be connected via a path, all of which is in $\tilde\Gamma$, which has length bounded by a fixed constant multiple of the Euclidean distance between $x$ and $y$.


\subsection{Acknowledgements}
The authors would like to thank the Centre de Recerca Matem\`atica, Barcelona for holding a conference where they developed some of the main ideas for this paper, and John Garnett for his helpful advice. The second author was supported in part by NSF DMS 0502747
and NSF DMS 0800837 (renamed to NSF DMS 0965766).

\subsection{Animation}
The first author created some animation exemplifying the construction in this paper.  It is available at\\ \url{http://www.math.sunysb.edu/~schul/math/AzzamSC-link.html}


\section{Notation and tools}\label{s:notation-and-tools}

\subsection{Notation}

Let $|A|$ denote the diameter of a set $A$.
Let $\cH^1(A)$ denote the 1-dimensional Hausdorff measure of $A$ and 
for a curve $\gamma$, let $\ell(\gamma)$ denote the arclength of $\gamma$.
See \cite{Mattila} for a discussion of Hausdorff measure and arclength.  
For a set $A\subseteq\bR^{d}$, define
\[A_{\delta}=\{x\in\bR^{d}:\dist(x,A)<\delta\}.\]
For points $x,y\in\bR^d$, and $\rho\geq 0$, we will define
\[R_{\rho}(x,y):= B(\frac{x+y}{2},\frac{1+\rho}{2}|x-y|)\]
and let $R(x,y):=R_{0}(x,y)$. Also define
\[S_{\lambda}(x,y)=B(x,(1-\lambda)|x-y|)\cap B(y,(1-\lambda)|x-y|)\]
and let $S(x,y):=S_{0}(x,y)$.
See Figure \ref{fig:lambda}.

For a ball $B$, and a set $K$, define the  Jones$-\beta$ number, $\beta_K(B)$ by setting $\beta_{K}(B)|B|$ to be the width of the smallest tube containing $K\cap B$, i.e.
$$\beta_K(B):=
	\inf\limits_{L \textrm{ line}}\ \ \sup\limits_{x\in B\cap K}\frac{\dist(x,L)}{|B|}\,.$$
We often omit the subscript $K$ when it is clear from context.
For $M>0$, we let $MB$ denote the ball with the same center but diameter $M|B|$. 
See Figure \ref{fig:beta}.

\pic{3in}{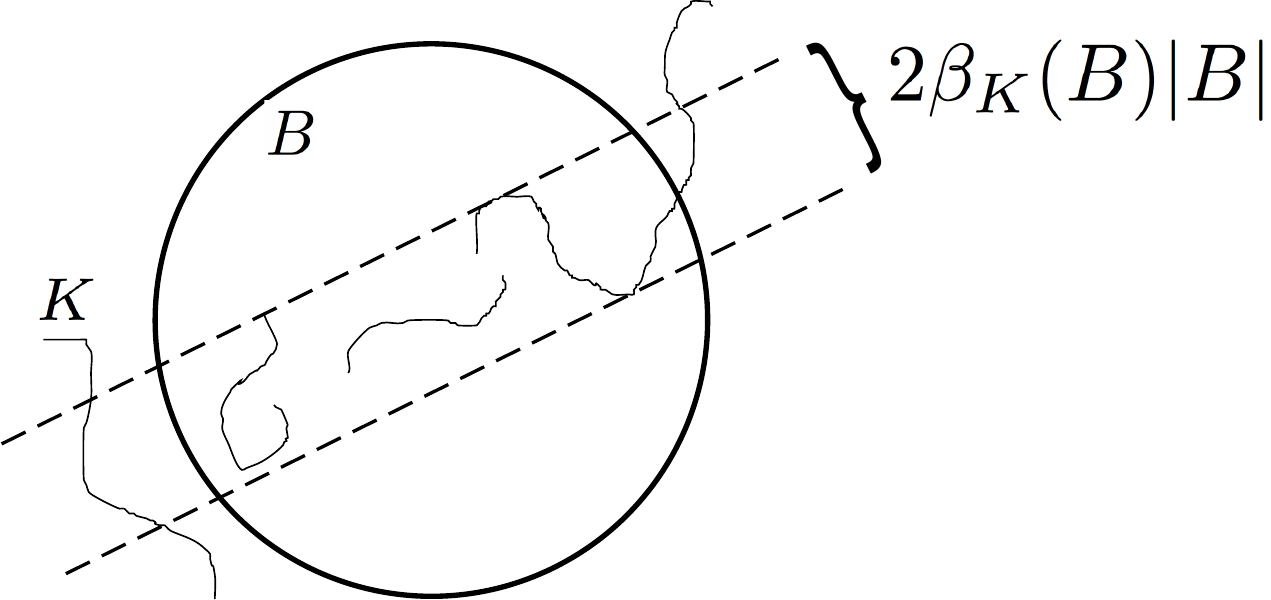}{Geometrically, $2\beta_{K}(B)|B|$ is the width of the smallest tube containing the portion of $K$ inside the ball $B$.}{beta}


If $\gamma$ is a curve with initial and terminal points $x$ and $y$, we say that  $\gamma$ is a {\it chord-arc path with constant $C$}, if its arclength parametrization is a $C$-bilipschitz function.
If we do not specify $C$, we assume it is obvious from the context.  In this paper, we will be constructing chord-arc-paths with constant $C_2$, where $C_{2}$ is a sufficiently large constant to be determined later. 

Let $\Gamma$ be a connected set containing $K$. We may assume $\cH^{1}(\Gamma)<\infty$.

\subsection{Cones}

For a constant $\alpha>1$ and for any point $\xi$ in a set $A$, define $C_{\alpha,A}$ the $(\alpha,A)$ {\it cone with apex} $\xi$ to be the union of connected components of the set
\begin{equation}
\{z\in\bR^d:|z-\xi|<\alpha \dist(z,A)\}
\label{eq:cone}
\end{equation}
which contain $\xi$ in their closures (the case of their being more than one such component is most evident in two dimensions when, say, $A$ is a circle, although this may still occur in higher dimensions;  see Figure \ref{fig:whyweneedcones}). We will let $C_{\alpha}(\xi)=C_{\alpha,\Gamma}(\xi)$. 

\subsection{Nets, Grids, and Cubes}

Let $\{\Delta_{k}\}_{k>0}$ be an increasing sequence of maximal $2^{-k}$-nets in $\Gamma$, and assume $|\Gamma|$ is small enough so that $\Delta_{0}=\{\xi_{0}\}$. Such a sequence may be constructed via induction on $k$.

We also create a lattice in the complement of $\Gamma$ that mimics a Whitney decomposition. Let $k_{0}$ be an integer to be chosen later. Let $\cN_{0}$ be a $2^{-k_{0}}$-net for $\Gamma_{1}\backslash\Gamma_{1/2}$, define $\cN_{k}'$ to be 
a~$2^{-k-k_{0}}$-net for $\bigg((\Gamma_{2^{-k}}\backslash\Gamma_{2^{-k-1}})\bigcup \cN_{k-1}\bigg)$ containing $\cN_{k-1}$.
Let $\cN_{k}=(\Gamma_{2^{-k}}\backslash\Gamma_{2^{-k-1}})\cap \cN_{k}'$. Let $\cN=\bigcup \cN_{k}$. The set $\cN$ forms the vertices of a ``grid" in the complement of $\Gamma$ upon which we will build our bridges by constructing polygonal paths between nearby points in $\cN$. This will ensure that the angles between segments in each path don't become too small. We note that, for $k_{0}$ large enough, we may ensure that
\begin{equation}
x\in \cN_{k}\mbox{ and }y\in \cN\cap B(x,10\cdot  2^{-k-k_{0}}) \Rightarrow y\in \cN_{k-1}\cup \cN_{k}\cup \cN_{k+1}.
\label{eq:close}
\end{equation}
Also note that for all $k>0$, every point in $\Gamma_{2^{-k}}$ is within $2^{-k-k_{0}+2}$ of a point in $\cN$.
Let $\cB_{0}=\{B_{0}\}=\{B(\xi_{0},1)\}$ and for $k>0$, 
\[\cB_{k}:=\{B(\xi,2^{-k}):\xi\in \Delta_{k}\},  \mbox{ and }  \cB=\bigcup \cB_{k}.\] 
Note that $B\in \cB_{k}$ implies $\frac{1}{2}B\in \cB_{k+1}$.

We need a version of dyadic cubes in the spirit of Michael Christ or Guy David. We do not have an underlying measure, so we cannot appeal to their constructions, however we can use ideas from \cite{Christ-T(b), David-T(b)}.
We fix a constant $J=100$, and  give a family (i.e. tree) structure on  $\cup_{k} \Delta_{kJ}$.
For each $x\in\Delta_{kJ}$ where $k>0$, we define a unique parent $y\in\Delta_{(k-1)J}$, 
so that $|y-x|$ is minimized.  If there is more than one such possible $y$, choose arbitrarily.
By the construction of $\Delta_{kJ}$, we have that $2^{-kJ}\leq |y-x|<2^{-kJ+1}$.
Let $D(x)$ be the collections of descendants of $x$ by the above family relation, 
and set $D_j(x)=D(x)\cap \Delta_j$, where $j$ here satisfies $j=Jl$ for some $l\geq 0$.

For $k\geq 0$, and $x\in\Delta_{kJ}$, let 
$$Q^o(x)=\bigcup\limits_{l\geq 0}\bigcup_{z\in D_j(x)\atop j=(l+k)J}B(z,2^{-j-100})\,,$$
and let $Q(x)$ be the closure of $Q^o(x)$. Let $\cQ_{k}=\{Q(x):x\in \Delta_{k}\}$ and $\cQ=\cup_k \cQ_k$. 
We have have the  properties described below.
\begin{lemma}\label{Christ-type-cubes-properties}
For $k\geq 0$ we have the following. 
\begin{itemize}
\item [(i)] $\Gamma=\cup_{x\in\Delta_{kJ}} Q(x)$.
\item [(ii)] If $x_1,x_2\in  \Delta_{kJ}$, then $Q^o(x_1)\cap Q^o(x_2)=\emptyset$.
\item [(iii)] If $x\in  \Delta_{kJ}$ then 
$B(x,2^{-kJ-100})\subset Q^o(x)\subset Q(x)\subset B(x,2^{-kJ}(1+2^{-J}))\,.$
\item  [(iv)]If $l\geq k$, $y\in \Delta_l$,  and $Q^o(x)\cap Q^o(y)\neq\emptyset$ then
$Q^o(x)\supset Q^o(y)$.
\end{itemize}
\end{lemma}
We note that if $\Gamma\subset \bR^d$ then the quantity $\#D_j(x)$ grows exponentially in $j$, with fixed base, depending on the dimension $d$.  We will not make use of this fact, nor will we need any bound on this growth except under special circumstances, where we will  have explicit bounds which will be independent of $d$.
\begin{proof}
First, (i) follows from the definition and induction on $k$.
To see (ii), suppose that $y\in  Q^o(x_1)$ and  
assume 
that $y\in D(x_1)$.
Then, $\dist(x,y)\leq \sum_{p=0}^{l-1} 2^{-pJ}2^{-kJ}\leq 2^{-kJ}(1+2^{-J})$.
Furthermore, if $x'\in D(x_1)\cap D_{(k+1)J}$ with $y\in D(x')$, then 
$$\dist(x',x_2)\geq \dist(x',x_1)\,,$$ and by the triangle inequality, 
$2\dist(x',x_2)\geq \dist(x_1,x_2)\geq 2^{-kJ}$ giving us
$\dist(x',x_2)\geq 2^{-kJ-1}$
Also, as above, $\dist(x',y)\leq 2^{-kJ-J}(1+2^{-J})$.  Using the triangle inequality again,
$$\dist(y,x_2)\geq  2^{-kJ-1} - 2^{-kJ-J}(1+2^{-J}) > 2^{-kJ-1}(1-2^{-J})$$ 
for $J>2$, for example $J=100$.
This gives that $y\notin B(x_2,2^{-kJ-100 + 1})$.  
We can run the same argument for any other ball in the definition of $Q^o(x_2)$.
In particular, we get (ii) by the density of $D(x_1)$ in $Q^o(x_1)$. Similarly, (iii) and (iv) follow as well.
\end{proof}

For a cube $Q$, denote by $MQ$ the set $\{x:\dist(x,Q)\leq M \diam(Q)\}$.
We will assume $M>2^{J+2}$.

\subsection{The Traveling Salesman Theorem}

The last tool we use is the following theorem:
\begin{theorem}[(Analyst's) Traveling Salesman Theorem] For a set $K$ in a Hilbert space $\cH$, define
\[ \beta_{K}=|K|+\sum_{k}\sum_{B\in\cB_{k}}\beta_K^{2}(MB)|B|.\]
There is $M_{0}$ such that for $M>M_{0}$ and any set $K$, if $\beta_{K}$ is finite, then $K$ may be contained in a connected set $\Gamma$ such that
\[\cH^{1}(\Gamma)\lec \beta_{K}.\]
Moreover, if $\Gamma$ is any rectifiable set of finite length, then 
\[\beta_{\Gamma}\lec \cH^{1}(\Gamma)\]
for any $M>1$.
\label{thm:TST}
\end{theorem}
Note that this imples that if $\Gamma\subset\cH$ is a connected set, then 
\[\beta_{\Gamma}\sim \cH^{1}(\Gamma)\]
This theorem was originally proved for $\cH=\bC$ by Peter Jones \cite{Jones-TSP}, then  generalized to $\cH=\bR^{d}$ by Kate Okikiolu \cite{O-TSP}, and to infinite dimensional Hilbert spaces by the second author of this paper \cite{Schul-TSP}.

We now begin the proof of the main theorem by showing we may contain $\Gamma$ in a set $\tilde{\Gamma}$  satisfying 
$$\cH^{1}(\tilde \Gamma)\lesssim \cH^{1}(\Gamma) + \beta(\Gamma)\lesssim \cH^{1}(\Gamma)\,.$$

\begin{remark}\label{r:tsp_extras}
By the proof of the Traveling Salesman Theorem, we may assume, by allowing an increase of $\cH^{1}(\Gamma)$ by a constant multiple, that  $\Gamma$ satisfies the  following properties for balls 
$B$  with center in $\Gamma$ and  $\beta(2B)<\ve$, with $\ve$ sufficiently small
\begin{itemize}
\item  There is a component of $\Gamma\cap B$ with diameter at least $|B|(1-\ve)$. 
\item The Hausdorff distance between $\Gamma\cap B$ and $L\cap B$  is bounded by 
$4\ve |B|$ for some  affine line $L$.
\end{itemize}
Furthermore, if $\Gamma$ had initially been $\ve-$Reifenberg flat with holes, then the above may be achieved while keeping $\Gamma$ $2\ve-$Reifenberg flat; if $K$ is $\ve-$Reifenberg flat with holes, then one may construct $\Gamma\supset K$ $2\ve-$Reifenberg flat and such that $\cH^{1}(\Gamma)\lesssim \cH^{1}(\Gamma')$ for any $\Gamma'\supset K$.
Henceforth, we shall assume $\Gamma$ has these properties. 
\end{remark}

\subsection{Outline}

Let us give a rough idea of our plan.
The proof of the theorem is a stopping time process run on a family of balls centered along $\Gamma$. The idea is that when a certain stopping time function becomes too big on one of the balls, this tells us to build a bridge between points. At first, it would seem that we merely have to check when the $\beta$-number of a ball was too large since this would detect a bend in the curve $\Gamma$ where one should build a short-cut. 
However, this doesn't account for the case of  sets which have small $\beta$ on all scales but contain a lot of length. (It is amusing to note that if we didn't have the assumption that $\Gamma$ had finite length, then $\Gamma$ could possibly satisfy $\beta(MB)\sim \ve$ for all sufficiently small balls $B$ with centers in $\Gamma$, in which case all balls will be ``flat" or have small $\beta$, but $\Gamma$ will have dimension at least $1+c\ve^{2}$ for some constant $c$ independent of $\ve$; see \cite{BJ97} or exercises in chapter X of \cite{Harmonic-Measure}.  A simple example of such a set is a flat Von-Koch snowflake). 
Therefore, it is necessary to keep a history of the $\beta$-numbers through the stopping time process, that is, not only do we keep track of the $\beta$ of a ball but also of the balls in the previous generations containing it (see condition \eqref{eq:Reifenberg-stop} below). We run the stopping time process until a chain of balls have accumulated a large total amount of $\beta$-numbers, and in this event we add a bridge. 
Separate treatment is given to balls with $\beta(MB)$ bounded away from zero.
\section{Constructing shortcuts}\label{s:constructing-shortcuts}

In this section we will classify all balls into three classes and explain how we build bridges in each of those cases. We will record some of their properties, and use those in Section \ref{s:route_finding}.  

\subsection{The Bridges}

The general idea for building a bridge between two points in $\Gamma$ inside a ball $B$ is to pick a point $z\in \Gamma^{c}$ whose distance from each of those points is $\sim |B|$ and then connect it to both of those points. This is not as trivial as it sounds. 
If one is not careful enough, it may be the case that after adding all our bridges to $\Gamma$ to form $\tilde{\Gamma}$, while each pair of points in $\Gamma$ may be joined by a path of small relative length, points between the bridges themselves may have to travel a long relative distance to reach each other. To see this, imagine two bridges connecting two different pairs of points in $\Gamma$, but their middles being very close (i.e. they form a narrow overpass). 
Building our bridges as polygonal paths with vertices in $\cN$ will help guarantee that points on different bridges can only be as close as their distance from $\Gamma$. 

\begin{lemma}
Suppose $\xi\in \Gamma$, $z\in \cN_{k}\cap B(\xi, C2^{-k})\cap C_{\alpha}(\xi)$, $\xi$ is in the closure of some component $A$ of $C_{\alpha}(\xi)$, and $z$ and $\xi$ may be joined by a path $p$ in $A$. Moreover, suppose $p$ has the property that for any ball $B=B(w,2^{-j-k_{0}+1})$ with $w\in \cN_{j}$ that intersects $p$, $2B\cap p$ is connected. Then there is a path $p'$ connecting $z$ and $\xi$ with the following properties:
\begin{itemize}
\item[a.] $p'$ is a polygonal path in $C_{2\alpha}(\xi)$ with vertices in $\cN$
\item[b.] $\ell(p')\lec \ell(p)$,
\item[c.] if $[x,y]$ is an edge in the path $p'$ and $x\in \cN_{j}$ for some $j$, then $y\in \cN_{j-1}\cup \cN_{j}\cup \cN_{j+1}$ and there is $\lambda>0$ such that
\[R_{\lambda}(x,y)\cap (\cN\backslash\{x,y\})=\emptyset\,,\]
and
\item[d.] $|x-y|\sim 2^{-j-k_{0}}$.
Here, $\lambda$ and all other implied constants are universal.
\end{itemize}
\label{thm:bridgelemma}
\end{lemma}

\begin{proof}
Since $\{B(w,2^{-j-k_{0}+1}):w\in \cN_{j}, j\geq 0\}$ is a cover of $\Gamma^{c}$, consider the subcollection $\cC$ of all those balls that intersect $p\backslash\{\xi\}$. Let $\gamma:[0,\ell(p))\rightarrow p\backslash\{\xi\}$ be the arclength parametrization of $p\backslash\{\xi\}$, Let $\cI=\{\gamma^{-1}(2B):B\in \cC\}$. Choose a subcollection $\{I_{j}\}\subseteq \cI$ so that the rightmost endpoint of $I_{j}$ is contained in $I_{j+1}$ and so that $I_{0}=\gamma^{-1}(B(z,2^{-k-k_{0}+2}))$, and so that no point in $[0,\ell(p))$ is contained in more than two sets in $\{I_{j}\}$. Hence, $\sum \ell(I_{j})\lec \ell(p)$. Let $B_{j}\in \cC$ be the ball so that $I_{j}=\gamma^{-1}(2B_{j})$ and let $w_{j}$ be their centers, with $w_{0}=0$. Then $2B_{j}\cap 2B_{j+1}\neq\emptyset$ and by $\eqref{eq:close}$, $|w_{j}-w_{j+1}|\lec |B_{j}|$. Moreover, 
\[\sum |B_{j}|\lec \sum \ell(I_{j})\lec \ell(p).\]  
Let $p'$ be the path $\bigcup_{j=0}^{\infty}[w_{j},w_{j+1}]$. Then
\[\ell(p')\leq \sum|w_{j}-w_{j+1}|\leq \sum 2|B_{j}|\lec \ell(p).\]

\pic{5in}{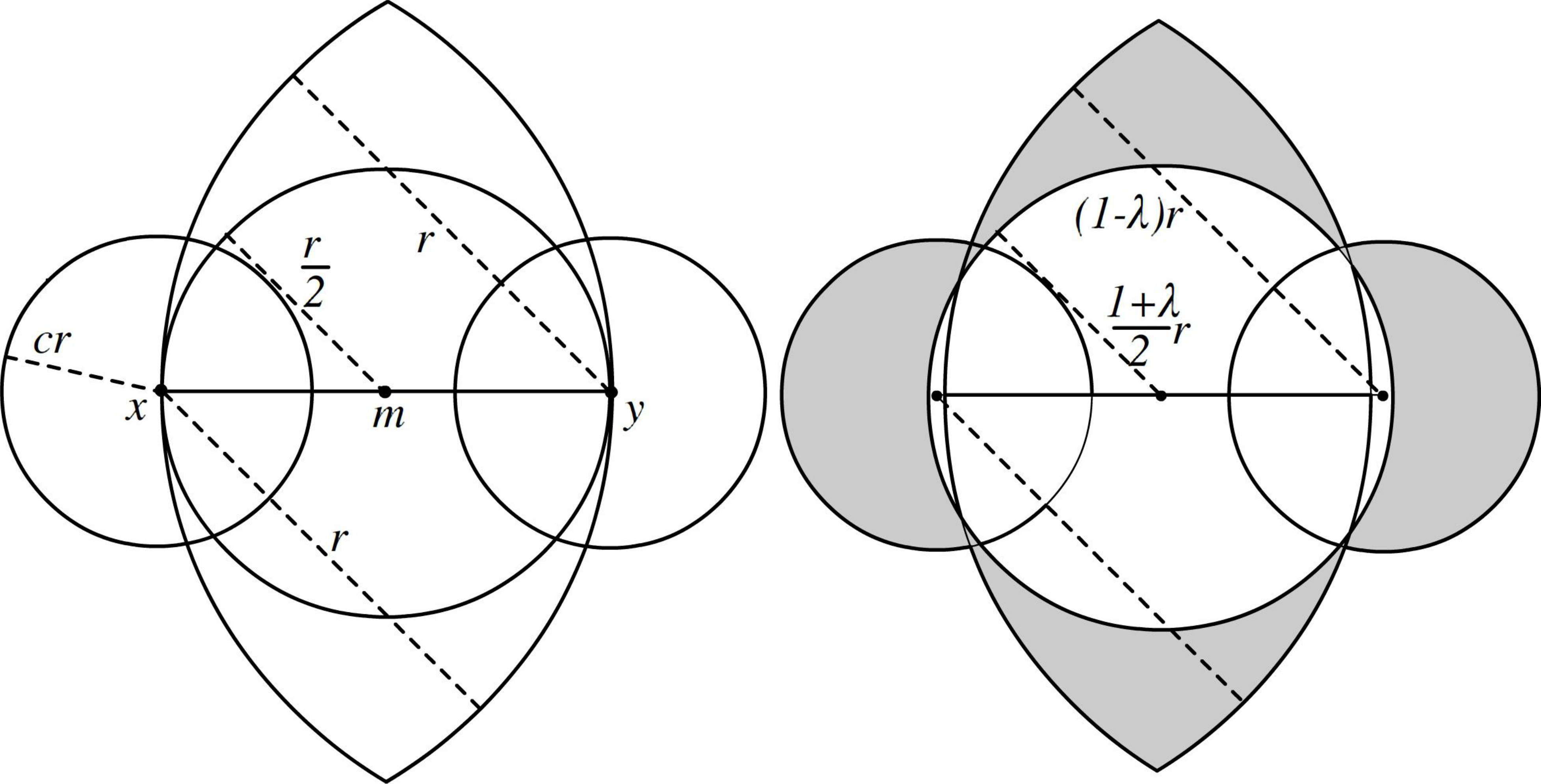}{On the left are the balls $B(x,cr),B(y,cr),R(x,y)$, and the set $S(x,y)$. Then by slightly adjusting $r$ in the definition of the latter two sets, we know there is $\lambda$ small such that $R_{\lambda}(x,y)$ is contained in the shaded region $B(x,cr)\cup B(y,cr)\cup S_{\lambda}(x,y)$.}{lambda}

Hence we can find a polygonal path so that (b) is satisfied. We will now adjust this path so that (b) is still satisfied but so that (c) is true. Let $[x,y]$ be an edge in $p'$. If $x\in \cN_{j}$, by the work above, $y\in B(x,2^{-j-k_{0}+3})$ and hence is in $\cN_{j-1}\cup \cN_{j}\cup \cN_{j+1}$ and $|x-y|> 2^{-j-k_{0}-1}$. Let $r=|x-y|$, then by the previous sentence there is a constant $b=b(k_{0})<1$, such that $(B(x,br)\cup B(y,br))\cap \cN=\emptyset$. 
By some planar geometry, (see Figure \ref{fig:lambda}),
there exists a small constant $\lambda=\lambda(b)>0$ such that
\[R_{\lambda}(x,y)\subseteq B(x,cr)\cup B(y,cr)\cup S_{\lambda}(x,y).\]
Suppose that $[x,y]$ didn't satisfy (c). Let $w\in R_{\lambda}(x,y)\cap (\cN\backslash\{x,y\})$. Replace the edge $[x,y]$ with $[x,w]$ and $[w,y]$. The total length we have added is no more than $[x,y]$, and moreover, since $w\not\in B(x,br)\cup B(y,br)$, we have $|w-x|,|w-y|<(1-\lambda)|x-y|$. 
Repeat the process on these two new edges, checking to see if they satisfy (c) and replacing them if not. 
This replacement can only happen a finite number of times, since the vertices of any new edge we add must be in $\cN_{j-1}\cup \cN_{j}\cup \cN_{j+1}$ by \eqref{eq:close}, but their mutual distances are decreasing by a factor of $1-\lambda$ each time we add a new edge. By doing this on each edge in $p'$ a finite number of times, we have adjusted $p'$ into a path that satisfies (c) and increased it's length by no more than some universal factor.

Suppose $|x-y|$ is an edge with $x\in \cN_{j}$. We have already seen that $|x-y|\gec 2^{-k-k_{0}}$. If $|x-y|\geq 10\cdot 2^{-k-k_{0}}$, then $B(\frac{x+y}{2},2^{-k-k_{0}+2})\cap\cN$ is empty by part (c), but this contradicts the sentence following \eqref{eq:close}.

Finally, if $k_{0}$ is large enough (depending on $\alpha$), the final product $p'$ will be contained in $C_{2\alpha}(\xi)$, which gives (a).
\end{proof}

\begin{remark}
We will only replace paths $p$ with polygonal paths $p'$ when $p$ satisfies the conditions of this lemma. In fact, it so happens that the only paths $p$ we ever have to deal with are polygonal paths composed of either one segment or two segments that make an angle of $\frac{\pi}{4}$. Hence, when we refer to a path or polygonal path, we will assume it has the properties in the lemma. 
\end{remark}

\subsection{Three types of cubes}
We will classify all cubes in $\cQ$ as \typezero, \typeone, or \typetwo.
Let $\delta,\ve>0$ and $M>0$, to be chosen later
(see  subsection \ref{s__main-lemma}).
If $B\in\cB$ satisfies $\beta(MB)>\delta\ve$ we say that $B$ is of {\it \typetwo}. If $Q=Q(x)\in \cQ_{k}$ is such that $B(x,2^{-kJ})$ is \typetwo, call $Q$ {\it \typetwo} as well. The rest of the cubes are divided into two distinct classes:\typezero and \typeone.  The class \typeone will be defined in the following section, and the class \typezero will simply be all the cubes in $\cQ$  which are not of the types \typeone or \typetwo.


\subsubsection{Definition of and construction at \typeone cubes}

Let $\ve>0$ and $M>0$, to be chosen later, with $\ve\ll\frac1M$
(see  subsection \ref{s__main-lemma}).
Suppose $\Gamma$ has diameter small enough so that $\beta(Q_{0})<\ve$, where $Q_{0}=Q(\xi_{0})$, $\Delta_{0}=\{\xi_{0}\}$. Call $\{Q_{j}\}_{j=m}^{k}$ a {\it chain} if $Q_{j}\in\cQ_{j}$, and $Q_{j+1}$ is a child of $Q_{j}$ for each $j\geq m$.

Going through the cubes in order, consider the first cube $Q$  (if it exists) such that there is a chain $\{Q_{j}\}_{j=0}^{k}$ so that $Q_{k}=Q\in \cQ_{k}$, $\beta(Q_{j})<\ve\delta$ for $0<j\leq k$ and
\begin{equation}
\sum_{j=1}^{k-1} \beta(MQ_{j})^{2}\leq \ve< \sum_{j=1}^{k} \beta(MQ_{j})^{2}.
\end{equation}
This sum is essentially a truncated Jones function (see \cite{Harmonic-Measure}). Call $Q_{k}$ {\it \typeone}. Let $B=B_{k}=B(x,2^{-kJ})$ where $x$ is the center of $Q$. Pick $z\in 2B_{k}\cap \cN_{kJ}$ closest to the center of $B$.

For each $\xi \in 2B\cap \Delta_{kJ}$ (there will be at most three of these for $\ve$ small), pick $\xi'\in B(\xi,M\ve|B|)\cap \Gamma$ that is closest to $z$, and note that, for small $\ve\ll\frac{1}{M}$, $[z,\xi']\subseteq C_{\alpha}(\xi')$, so by Lemma \ref{thm:bridgelemma}, we may connect $z$ to $\xi'$ by a polygonal path ($\alpha$ will be fixed in Section \ref{s:route_finding}). See Figure \ref{fig:okaybridge}.

\pic{4.5in}{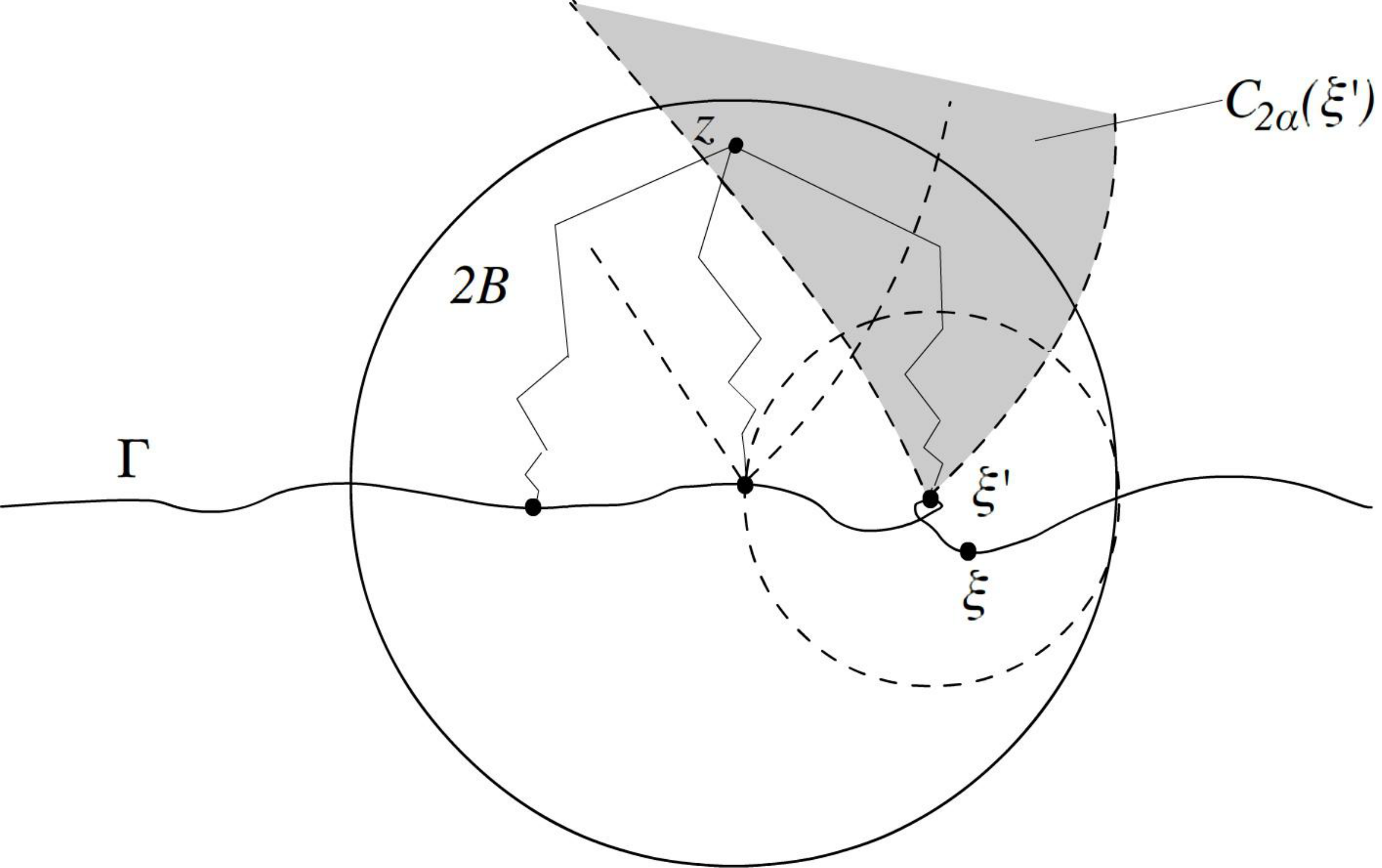}{Building bridges for \typeone balls.}{okaybridge}

\begin{remark}
Building our path within a component of $C_{2\alpha}(\xi')$ in this way ensures that the path won't intersect $\Gamma$ at too sharp an angle or get too close to $\Gamma$ before reaching $\xi'$, and for this reason it is in fact necessary for any curves we add to be contained in cones centered on $\Gamma$.

It is also important to note why we can't necessarily connect $z$ directly to $\xi$ instead of the nearby point $\xi'$. In $\bR^{2}$ this is evident since the point $\xi'$ may be separated from $z$ by $\Gamma$ itself, making it impossible to connect $z$ to $\xi$. In higher dimensions, it may be the case that we can connect $z$ to $\xi$ by a polygonal path, but possibly not without getting too close to $\Gamma$, which may be the case if $\Gamma$ resembles a Peano curve near $\xi$. In other words, if we had simply taken $C_{\alpha}(\xi)$ to be the entire set in \eqref{eq:cone}, it wouldn't always be possible to build our path in $C_{\alpha}(\xi)$ since whatever component we start building it in may not contain $\xi$ in it's closure (see Figure \ref{fig:whyweneedcones}). Hence, we will have to make due with building bridges to points $\xi'$ near $\xi\in\Delta_{k}$. 

\end{remark}

\pic{4in}{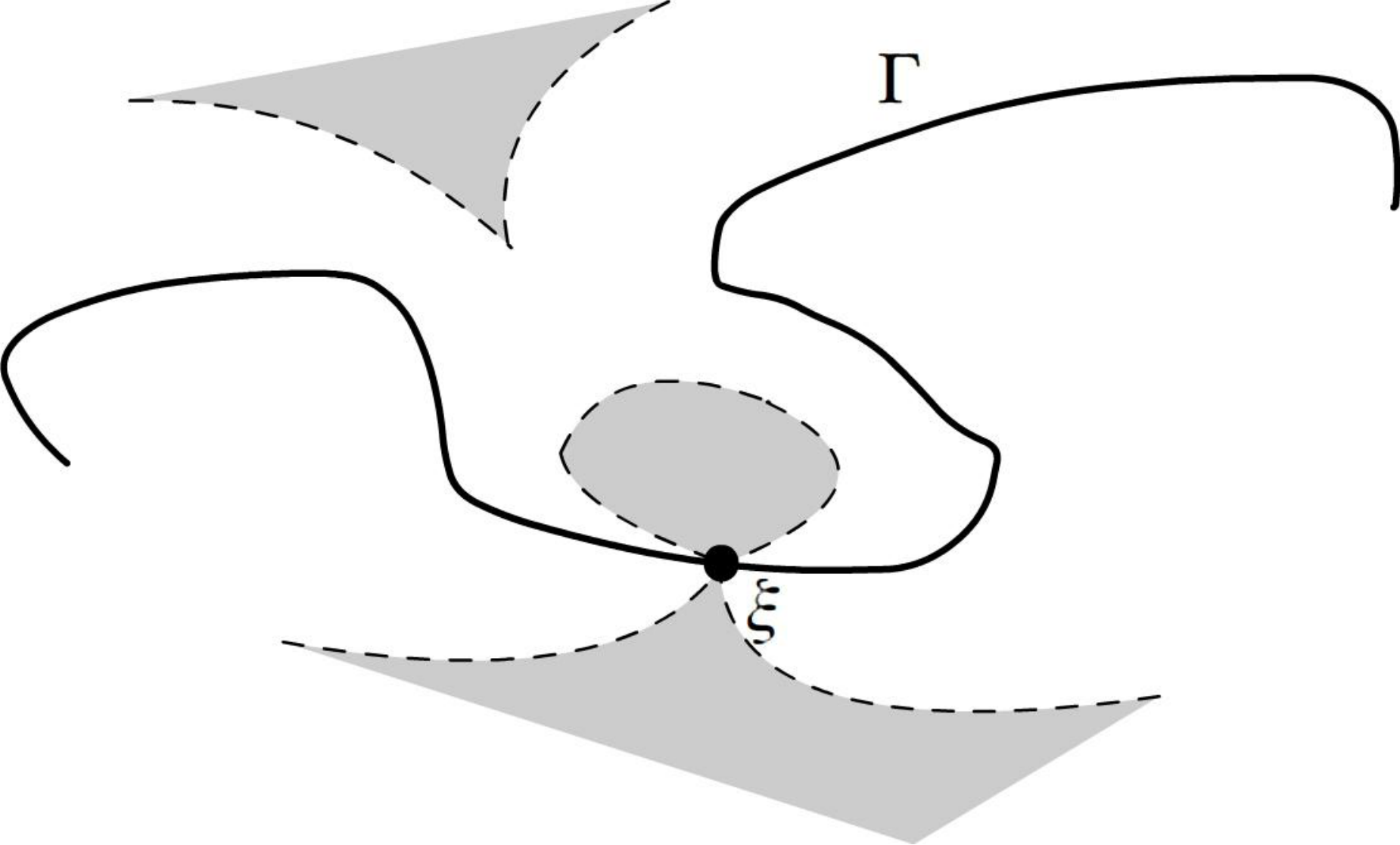}{The shaded regions depict the set $\{z:|z-\xi|<\alpha d(z,\Gamma)\}$, which may have several components, some of which may not contain $\xi$ in their boundaries.}{whyweneedcones}

We continue going through the tree in this manner. 
We stop and declare $Q_{k}$ to be \typeone every time we have 
$Q_{k}$  satisfying 
\begin{equation}
\sum_{j=m+1}^{k-1} \beta(MQ_{j})^{2}\leq \ve< \sum_{j=m+1}^{k} \beta(MQ_{j})^{2}
\label{eq:Reifenberg-stop}
\end{equation}
for a chain of cubes $\{Q_{j}\}_{j=m}^{k}$, so that $\{Q_{j}\}_{j=m+1}^{k-1}$  are all not of type \typetwo, 
where $Q_{m}$ is one of:  $Q_{0}$, a \typetwo cube,  or the previous \typeone; %

Let $z$ be any point in $2B_{k}\cap \cN_{kJ}\cap C_{\alpha}(\xi)$ closest to the center of $B_{k}$. For each $\xi\in \Delta_{kJ}\cap 2B_{k}$, we pick points $\xi'\in B(\xi,M\ve 2^{-k})$ closest to $z$ and connect them to $z$ as before.

We remind the reader that 
if a cube is neither \typeone or \typetwo, call it {\it \typezero}.  Note that all the intermediate cubes 
$\{Q_{j}\}_{j=m+1}^{k-1}$ in equation \eqref{eq:Reifenberg-stop} are \typezero.

\subsubsection{Definition of and construction at  \typetwo balls}
Let $\delta>0$, $k_{1}$, and $C>2$ be numbers to be specified later. 
We recall that  a ball $B=B(\xi_{1},2^{-k})$ is {\it \typetwo} if $\beta(MB)>\delta\ve$. 
In this case, consider all $\xi_{2}\in B(\xi_{1},C2^{-k})\backslash B(\xi_{1},2^{-k+1})\cap \Delta_{k}$ such that 

\begin{description}
\item[\star] $\exists \eta_{j}\in B(\xi_{j}, 2^{-k})\cap\Gamma$, $j=1,2$, and $z\in C_{C\alpha}(\eta_{1})\cap C_{C\alpha}(\eta_{2})\cap B(\xi_1,C2^{-k})\cap B(\xi_{2},C2^{-k})\cap\cN_{k+k_{1}}'\neq\emptyset$ that can be connected to each $\eta_{j}$ by a path $p$ with $\ell(p)\lec 2^{-k}$.
\end{description}

\pic{4.5in}{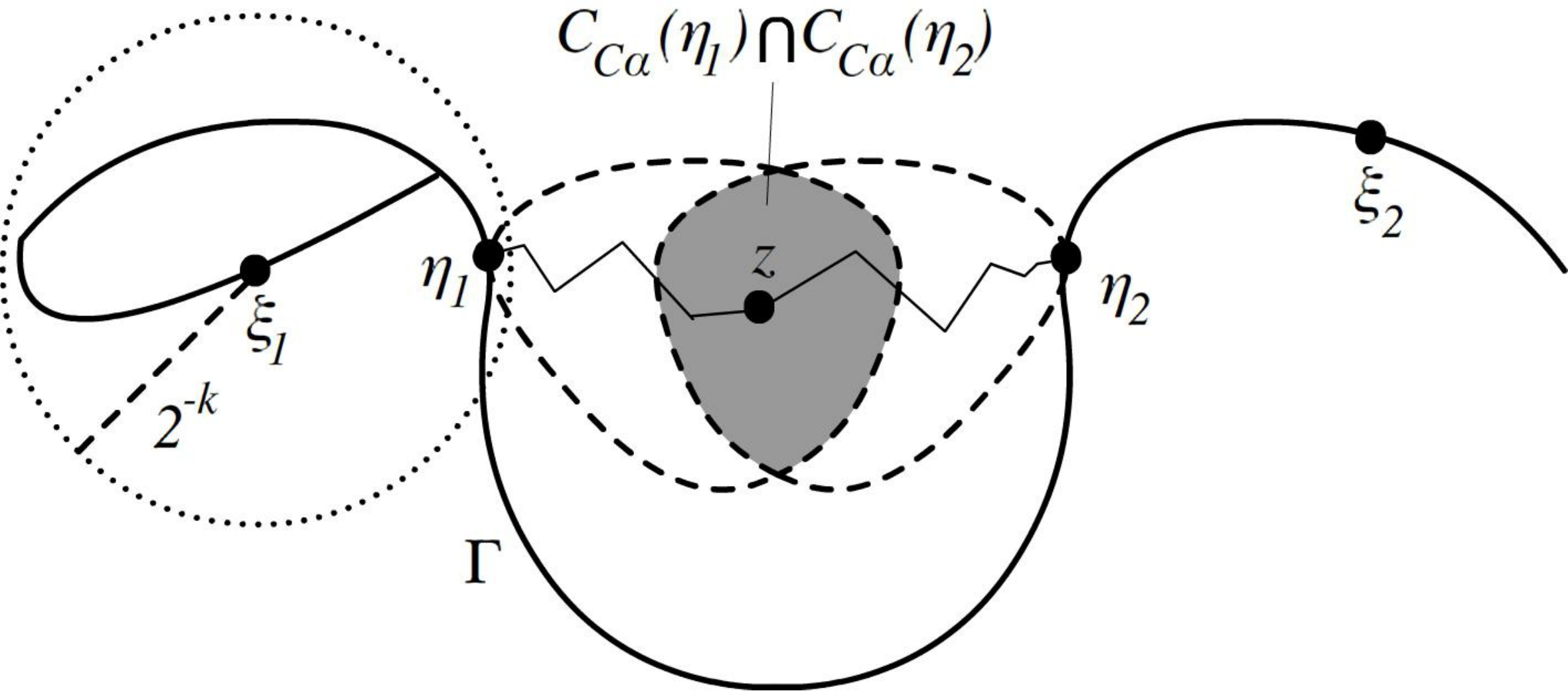}{Building bridges from a \typetwo ball.}{badbridges}

By Lemma \ref{thm:bridgelemma},
the condition implies we may connect the $\eta_{j}$ by a path of length $\lec 2^{-k}$. 
If we choose $\ve$ small enough, then if $B\subseteq NB'$, where 
$2^{-k}\leq |B'| \leq 2^{-k+1}$, $2\leq N\leq 2M$, and $\beta(NB')<\ve$ we may pick $\eta_{j}\in B(\xi_{j},M\ve 2^{-k})$ that satisfy \star as in the case of \typeone balls (since our set is so straight locally, the cones $C_{\alpha}(\eta_{j})$ will have large intersection for our choices of $\xi_{j}$). We then connect those points by a polygonal path. This exception will be needed in Case 2 of Lemma \ref{thm:thefinallemma}.

{
If a $B'$ as above does not exist for $B$, then just add a path connecting points $\eta_{1}$ and $\eta_{2}$ satisfying the condition \star. }

Note that such pairs might not exist in general, but we will show below that they exist often enough.

\begin{remark}
This is the only point in the proof that breaks down in infinite dimensions, since we are controlling the number of bridges we build for a ball $B\in\cB_{k}$ by $\# (\Delta_{k}\cap MB)$, which is uniformly bounded so long as we work in finite dimensions. We do however conjecture that Theorem \ref{t:main-theorem} still holds with constants independent of dimension and in fact holds in the case of an infinite-dimensional Hilbert space.
\label{thm:wesuck}
\end{remark}

\begin{remark}
We note that if, for example, one has the assumption of the set $K$ being $\ve-$Reifenberg flat (with holes), where $\ve$ is sufficiently small, then by the proof of Theorem \ref{thm:TST} (see \cite[Section 4, p. 365]{Schul-TSP}), $\Gamma$ can be constructed to be $C\ve$-Reifenberg flat ($C>0$ being some universal constant), and hence $\# (\Delta_{k}\cap MB)$ will be uniformly bounded. In this case, the proof of Theorem \ref{t:main-theorem} below gives Theorem \ref{t:revisit-reifenberg} with constants independent of the ambient dimension. In fact, in this setting, one can replace $\bR^{d}$ with an infinite-dimensional Hilbert space.
\label{remark:reifenberg-proof} 
\end{remark}

\subsection{Estimating the total length}
\label{s:estimating-total-length}
Let $\tilde{\Gamma}$ be the union of $\Gamma$ with all paths we have added for \typeone cubes and \typetwo balls.
The goal of this section is to bound the length of $\tilde\Gamma$.

\begin{lemma}
Let $\tilde{\Gamma}$ be as above. Then $\cH^{1}(\tilde{\Gamma})\lec_{d}\cH^{1}(\Gamma)$.
\label{thm:length}
\end{lemma}

To prove this, we will need some additional lemmas. The following lemma and it's techniques will be used to help estimate the length we have added on from \typeone cubes. It will also be used later, when we find short paths between points in $\tilde{\Gamma}$.

\begin{lemma}
\label{thm:betasmall}
There exists an $\ve_{0}>0$ and $C>0$ such that for any $K>0$ and any connected compact $\Gamma$ contained in any Hilbert space satisfying 
\begin{itemize}
\item[(i)] for any $r<|\Gamma|$ and $x\in \Gamma$, we have
$$\beta_\Gamma(B(x,r))< \ve_{0},$$ 
\item[(ii)] 
for any $r<|\Gamma|$ and $x\in \Gamma$, we have
that the Hausdorff distance between $B(x,r)\cap \Gamma$ and $B(x,r)\cap L$ is less than $4\ve_{0}|B|$ for some line $L$. 
\item[(iii)] for  $x\in \Gamma$
$$\sum_{k=-\log|\Gamma|}^\infty\beta^2_\Gamma(x,2^{-k}) < K\,,$$
\end{itemize}
we have 
$$\cH^{1}(\Gamma)\lesssim |\Gamma|e^{CK}.$$ 
\end{lemma}
Note that assumption $(ii)$ is assured by Remark \ref{r:tsp_extras}.
\begin{remark}
The assumptions $(i)$ and $(ii)$ above may be omitted, but that would require a longer, more complicated proof. This lemma does not appear in the literature for infinite dimensional Hilbert spaces.  However, for the plane, this is shown in \cite{BJ}. Also see Theorem X.2.1 in \cite{Harmonic-Measure} as well as \cite{Le}. 
\end{remark}
\begin{proof}
Without loss of generality, we assume $|\Gamma|=1$.
Fix $C'$ such that $\frac1{10 \ve}>C'\geq 1$.
We inductively construct a sequence of polygonal curves $P_n$ as follows.
Let $x_0,x_1\in\Gamma$ be such that $\dist(x_0,x_1)=|\Gamma|$.
We set $P_0=I_\emptyset=[x_0,x_1]$ and get $\ell(P_0)=1$.
By property (ii), we have the existence of $x_2\in \cN_{\ve_{0}}{[x_0,x_1]}$, the $\ve_{0}$ neighborhood of the interval $[x_0,x_1]$ satisfying
$|x_1-x_2|\sim_{C'}|x_0-x_2|$.
We set $P_1=I_0\cup I_1=[x_0,x_2]\cup[x_2,x_1]$.
By the Pythagorean theorem we have that 
$$\ell(P_1)=|x_2-x_0|+ |x_1-x_2| \leq 1 +C\beta_{\{x_0,x_2,x_1\}}^2(x_2,|x_0-x_1|)\,.$$
The constant $C$ here depends only on the choice of $C'$ above.
We will abuse notation and, for a triple $x_0,x_1,x_2$ as above,   denote   
$$\beta[x_0,x_1]:=\beta_{\{x_0,x_2,x_1\}}^2(x_2,|x_0-x_1|)\,.$$
We may now iterate this process on each of the intervals $[x_0,x_2]$ and $[x_2,x_1]$, getting
a polygonal curve $P_2$ satisfying
\begin{eqnarray*}
\ell(P_2)&\leq&
\ell(I_0) (1+C\beta^2[x_0,x_2]) + 
\ell(I_1)(1+C\beta^2[x_2,x_1])\\
&=&
\ell(P_1)\bigg(
\frac{\ell(I_0)}{\ell(P_1)}(1+C\beta^2(I_0)) + 
\frac{\ell(I_1)}{\ell(P_1)}(1+C\beta^2(I_1)) \bigg)\\
&\leq& 
\ell(P_0) (1+C\beta^2(I_\emptyset))\bigg(\sum_{i=0}^{1}
	\frac{\ell(I_i)}{\ell(P_1)}(1+C\beta^2(I_i))  \bigg)\,.
\end{eqnarray*}

Continuing inductively, we get by at the $n^{th}$ step a polygon with $2^n$ edges, satisfying
$$
\ell(P_n)\leq\sum_{\omega\in\{0,1\}^n} a_\omega \Pi_{k=0}^n (1+C\beta[I_{\omega_k}])$$
where $\omega_k$ is the truncation to the first $k$ elements of $\omega$, and
$1=\sum_{\omega\in\{0,1\}^n} a_\omega$.
Note that 
$$\Pi_{k=0}^n (1+C\beta[I_{\omega_k}])\leq e^{CK}\,.$$
Also note that, by property (i),  at the $n^{th}$ step, the vertices of $P_n$ form a $\bigg(\frac{1+C'}{C'}\bigg)^{-n}$ net for $\Gamma$.

\end{proof}

\begin{remark}
A few things things should be mentioned about the proof:

\begin{enumerate}
\item The choice of $x_{2}$ between $x_{0}$ and $x_{1}$ is not important so long as we pick it far from $x_{0}$ and $x_{1}$, i.e. $d(x_{2},\{x_{0},x_{1}\})\gec |x_{0}-x_{1}|$. This ensures that our sequence of paths will converge to $\Gamma$.

\item In the construction in the proof, we could have stopped iterating at a finite polygonal path, or more generally, cease adjusting our sequence of curves on some 
collection of segments. The resulting path, by virtue of being a polygon or having corners, would not satisfy the conditions of the theorem at the vertices, however it  would still satisfy the conclusion of the lemma. 

\item Also note that condition (iii) can be replaced by
\[\sum_{kJ>-\log |\Gamma|} \beta_{\Gamma}^{2}(x,M2^{-kJ})<K\]
since, for $M$ large enough and $\ve_0$ small enough, this will imply (iii) (with perhaps a different $K$).

We mention these facts since we will want to use the construction in this proof to construct polygonal paths with vertices in $\bigcup \Delta_{kJ}$ using the \typeone condition on cubes. 
\end{enumerate}
\label{thm:case2remark}
\end{remark}

\begin{lemma} 
Let $E(x,y)$ be the collection of maximal cubes $Q$ with $|Q|\leq |S(x,y)\cap \Gamma|$, centers in $\cnj{S(x,y)}$, and are not \typezero (so no cube in $E(x,y)$ is properly contained in another cube in $E(x,y)$). There exists $\ve_{0}>0$ such that for any $x,y\in \Gamma$, if $\beta(R_{x,y})<\ve_{0}$, then
 \[\sum_{Q\in E(x,y)}|Q|\lec |x-y|.\]
 \label{thm:bj}
\end{lemma}

\begin{proof}
Let $\Gamma_{x,y}=\cnj{S(x,y)}\cap \Gamma$. 
Let 
\[\cQ_{x,y}=\{Q\in \cQ: Q \mbox{ has center in }\Gamma_{x,y} \mbox{ and } |Q|\leq |\Gamma_{x,y}|\}.\] 
Let $k$ be the largest number for which there is no $Q\in \cQ_{k}$ contained in $\Gamma_{x,y}$. Construct a path as in Lemma \ref{thm:betasmall} as follows. For  $j\geq 1$, choose points $x_{j}$ and $y_{j}$ in $\Delta_{(k+j)J}\cap\Gamma_{x,y}$ closest to $x$ and $y$ respectively 
(so $x_{j}\rightarrow x$ and $y_{j}\rightarrow y$ as $j\rightarrow\infty$).

Note that $[x_{1},y_{1}]\cup \bigcup_{j\geq 1}[x_{j},x_{j+1}]\cup [y_{j},y_{j+1}]$ is a connected path connecting $x$ to $y$. Let $P_{0}=[x_{1},y_{1}]$

Pick $x'\in \Delta_{(k+1)J}\cap \Gamma_{x,y}$ closest to the midpoint of $[x_{1},y_{1}]$, replace $P_{0}$ with $P_{1}=[x_{1},x']\cup [x',y_{1}]$. Continue this process, suppose we have a path $P_{n}$ with edge $[a,b]$, $a,b\in \Delta_{(k+1)J}$ and there is a point $x'\in \Delta_{(k+1)J}$ between $a$ and $b$ closest to the midpoint between $a$ and $b$ (so $d(x',\{a,b\})\gec |a-b|$). Then replace that edge with $[a,x']\cup [x',b]$ to form $P_{n+1}$. If there is no such point between $a$ and $b$, then leave that edge. 

In the end, we have constructed a path $P^{1}$ with vertices all points in $\Delta_{(k+1)J}\cap\Gamma_{x,y}$. 
Repeat the process as follows. 
Take an edge $[a,b]$, $a,b\in \Delta_{(k+1)J}$ such that either $a$ or $b$ is the center of a cube $Q\in \cQ_{k+1}\cap \cQ_{x,y}$ that is \typezero and is not the child of a \typeone or \typetwo cube in $\cQ_{x,y}$. 
On this edge, perform the same as above, i.e. replace it with $[a,x']\cup [x',b]$ where $x'\in \Delta_{(k+2)J}\cap \Gamma_{x,y}$ etc. This gives $P^2$.
Continue inductively to get a sequence of paths $P^{n}$ that converge to a path $P_{[x_{1},y_{1}]}$. By Lemma \ref{thm:betasmall} and Remark \ref{thm:case2remark}, $\ell(P_{[x_{1},y_{1}]})\lec |x_{1}-y_{1}|$ if $\ve_{0}$ is small enough (so that (i) of Lemma \ref{thm:betasmall} is satisfied). 

We can do similarly for each segment $[x_{j},x_{j+1}]$ and $[y_{j},y_{j+1}]$ to make paths $P_{[x_{j},x_{j+1}]}$ and $P_{[y_{j},y_{j+1}]}$ respectively. Let 
\[P=[x_{1},y_{1}]\cup\bigcup_{j\geq 1}P_{[x_{j},x_{j+1}]}\cup P_{[y_{j},y_{j+1}]}\]
and 
\[\ell(P)\lec |x_{1}-y_{1}|+\sum_{j\geq 1}( |x_{j}-x_{j+1}|+|y_{j}-y_{j+1}|)\lec |x-y|\,,\]
where the last inequality is just the summing of a geometric sum.
Note that the center $a$ of any $Q\in E(x,y)$ is the endpoint of an edge $[a,b]\subseteq P$ of length $\ell([a,b])\sim |Q|$. Let $e_{Q}=(a,\frac{a+b}{2})$. Then $\{e_{Q}:Q\in E(x,y)\}$ is a disjoint collection of segments in $P$ with $\ell(e_{Q})\sim |Q|$, and thus
\[\sum_{Q\in E(x,y)}|Q|\lec \ell(P)\lec |x-y|,\]
which proves the claim.

\end{proof}

\begin{proof}[Proof of Lemma \ref{thm:length}]

Recall that any path $p$ added on for a \typetwo ball $B$ has length $\lec |B|$. 
Thus, by Theorem \ref{thm:TST}, the total lengths of all paths $p$ added may be bounded as follows.
\begin{align*}
\sum_{B\mbox{ \typetwo }}\sum_{p\mbox{ path for }B}\ell(p) 
 \lec \sum_{B\mbox{ \typetwo }}|B|
 & 
\leq \sum_{B\mbox{ \typetwo }}\frac{\beta(MB)^{2}}{(\delta\ve)^{2}}|B| \\
& \lec \cH^{1}(\Gamma).
\end{align*}

For a \typeone cube $Q$, let $p_{Q}$ be the path constructed for the ball $Q$ as described earlier. These have length $\lec |Q|$. Note that for each \typeone cube $Q$, there is a chain $A(Q)=\{Q_{m+1},...,Q_{k}\}$ with $Q_{k}=Q$, $B_{m+1},...,B_{k-1}$ all \typezero, and satisfying \eqref{eq:Reifenberg-stop}. 
Let $D(Q)$ be the set of \typeone descendants of $Q$ that are connected to $Q$ by a chain of \typezero cubes. Then
\begin{align*} \sum_{Q\mbox{ \typeone }}\ell(p_{Q}) 
& \lec \sum_{Q \mbox{ \typeone}}|Q| 
\leq \frac{1}{\ve}\sum_{Q \mbox{ \typeone}}\sum_{Q'\in A(Q)}|Q|\beta(MQ')^{2}\\
& =\frac{1}{\ve}\sum_{Q'\mbox{ \typezero }}\beta(MQ')^{2}\sum_{Q\in D(Q')}|Q|. \end{align*}

We claim that
\[\sum_{Q\in D(Q')}|Q|\lec |Q'|,\]
in which case the lemma will follow by Theorem \ref{thm:TST}. This follows by applying Lemma \ref{thm:bj} with $x$ and $y$ being points in $Q'$ that maximize $|x-y|$. 
\end{proof}


\section{Route finding: there are enough shortcuts}\label{s:route_finding}

We now turn to proving that $\tilde\Gamma$ is quasiconvex, that is, 
any two arbitrary points $x,y\in \tilde{\Gamma}$ may be connected by a curve in 
$\tilde{\Gamma}$ of length $\lesssim |x-y|$.  
This is the final step in proving  Theorem \ref{t:main-theorem} (as well as  Theorem \ref{t:revisit-reifenberg}).

In the first section, we  reduce to the case when $x,y\in\Gamma$, using the properties in Lemma \ref{thm:bridgelemma}. After that, we  state and prove the main lemma. Our main lemma  says that between any points $x,y\in\Gamma$ there are bridges connecting points that are almost collinear with $x$ and $y$. In the last part of this section, we  pull these results together and conclude the main theorem.

\subsection{The Case of $x$ or $y$ not in $\Gamma$}

\begin{lemma}
\label{l:segment-lemma}
There is a universal constant $a>0$ such that if $s_{1}$ and $s_{2}$ are two nonadjacent segments in $\tilde{\Gamma}\backslash \Gamma$, then $d(s_{1},s_{2})\geq a\cdot \min\{|s_{1}|,|s_{2}|\}$.
\end{lemma}

\begin{proof}
Let $[x,y]$ and $[z,w]$ be two non-adjacent segments such that $|x-y|\geq |z-w|$, $x,y,z,w\in \cN$, and suppose there are points $x'$ and $z'$ in each of these segments respectively such that $|x'-z'|=d([x,y],[z,w])<a|z-w|$, where $a$ is a small constant we will pick shortly. 
Set $r=|x-y|$. 

Recall that by the definition of $\cN$,  $w,z\not\in B(x,cr)$.
If $x'\in B(x,\frac{c}{2}r)$, we know
\[x\in R(s,t)\subseteq R(z,w)\subseteq R_{\lambda}(z,w)\]
where $(s,t)=[w,z]\cap B(x,cr)$, but this contradicts Lemma \ref{thm:bridgelemma}. Hence, applying a similar argument for $y$, we may assume 
\begin{equation}
d(x',\{x,y\})>\frac{c}{2}|x-y| 
\label{eq:away}
\end{equation}

The idea for the remainder of the proof is to shift the line $[w,z]$ so that it intersects $[x,y]$, in which case both lines will lie in a plane and we can prove the result more easily in this case. The general case will follow because the amount we needed to shift is very small since we are assuming that the lines come very close to each other. See Figure \ref{fig:segments}.

\pic{4in}{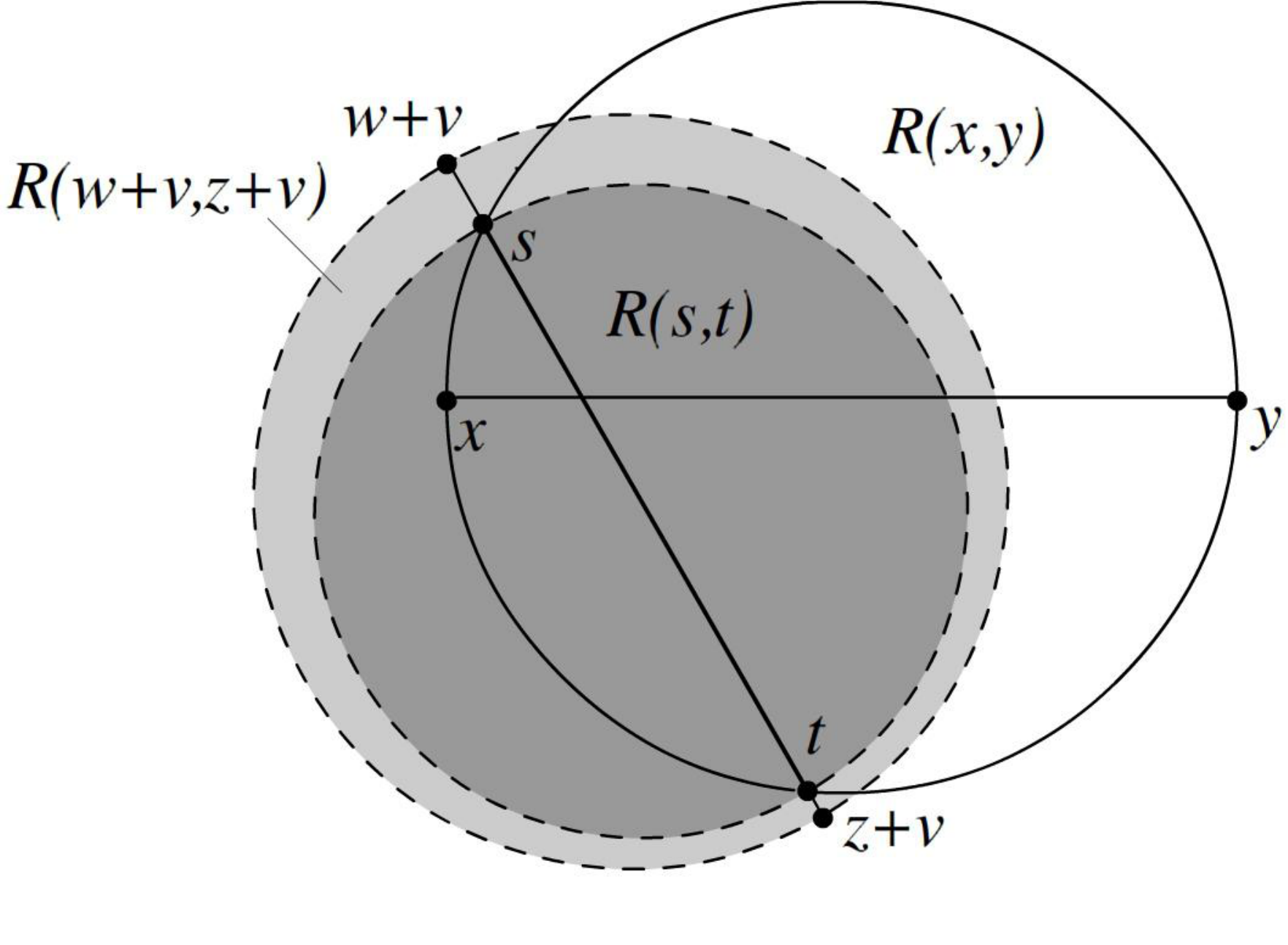}{Accompanying figure to Lemma \ref{l:segment-lemma}.}{segments}

Let $v=x'-z'$, so $|v|<a|z-w|$. By Lemma \ref{thm:bridgelemma}, $w,z\not\in R_{\lambda}(x,y)$, and if $a$ is much smaller than $\lambda$, $w+v,z+v\not\in R_{\frac{\lambda}{2}}(x,y)$ as well. Now let $(s,t)=[w+v,z+v]\cap R(x,y)$, which, for $a$ small enough with respect to $c$, is nonempty by \eqref{eq:away}. If $\frac{x+y}{2}\in (s,t)$, then 
\[|w-z|=|(w+v)-(z+v)|\geq 2\cdot \frac{1+\frac{\lambda}{2}}{2}|x-y|>|x-y|,\]
a contradiction. Hence $(s,t)$ avoids the midpoint of $[x,y]$. Suppose now that $x$ is closer to $(s,t)$ than $y$. Then for $a$ much smaller than $\lambda$,
\[x\in R(s,t)\subseteq R_{\frac{\lambda}{2}}(w+v,z+v)\subseteq R_{\lambda}(w,z),\]
again, a contradiction. 
\end{proof}

\begin{corollary}
To prove the main theorem, it suffices to show that if $x,y\in\Gamma$ then they are joined by a path of length $\lec |x-y|$. 
\label{thm:wlog}
\end{corollary}

\begin{proof}

If $x,y$ are in the same segment or two adjacent segments in $\tilde{\Gamma}$, then we're done. If they are in two different segments $s_{x}$ and $s_{y}$, then $|x-y|\geq a\min\{|s_{x}|,|s_{y}|\}$, but by \eqref{eq:close}, the endpoints of these segments are all of comparable distances between each other, so in fact $|x-y|\gec \max\{|s_{x}|,|s_{y}|\}$. 
Let $p_{x}$ and $p_{y}$ be shortest paths connecting $x$ and $y$ to points $x',y'\in \Gamma$ respectively. 
These paths will have lengths $\sim |s_{x}|$ and $\sim |s_{y}|$ respectively by the construction (and  Lemma \ref{thm:bridgelemma}), so $|x'-y'|\lec |x-y|$. 
Moreover, between $x'$ and $y'$, by assumption, there is a path $p$ connecting them of length $\lec |x-y|$, and thus the path $p_{x}\cup p\cup p_{y}$ connects $x$ and $y$ and has length
\[\ell(p_{x}\cup p\cup p_{y})\lec |s_{x}|+|x'-y'|+|s_{y}|\lec |x-y|+|x-y|+|x-y|\lec |x-y|.\]

Suppose $x\in \Gamma$ and $y\in \tilde{\Gamma}\backslash \Gamma$. There is a path of shortest length $p_{y}$ connecting $y$ to a point $y'\in \Gamma$ which has length $\ell(p_{y})\sim d(y,\Gamma))\leq \min\{|y-y'|,|y-x|\}$, and then a path $p_{xy'}$ connecting $y'$ to $x$ of length $\ell(p_{xy'})\lec |x-y'|$. If $|x-y'|<2|x-y|$, then
\[\ell(p_{y}\cup p_{xy'})\lec |x-y|+|x-y'|<3|x-y|\]
and otherwise, if $|x-y'|\geq 2|x-y|$, then
\[|x-y|\geq d(y,\Gamma)\sim |y-y'|\geq |x-y'|-|x-y|\geq |x-y|,\] 
thus 
\[|x-y|\sim d(y,\Gamma)\sim \ell(p_{y})\sim |y-y'|,\] 
hence
\[\ell(p_{y}\cup p_{xy'})\lec |x-y|+|x-y'|\leq 2|x-y|+|y-y'|\lec |x-y|\]
and that finishes the proof.

\end{proof}

\subsection{Main Lemma}\label{s__main-lemma}

\begin{lemma}[Main lemma]
Let $x,y\in \Gamma$. Then we may find a path $\gamma$ connecting $x$ and $y$ that is either a chord-arc path in $\Gamma$ or a union of segments $S$, with endpoints in $\Gamma$ and a set $P\subseteq\tilde{\Gamma}$ that satisfy
\begin{equation} 
\sigma=\sum_{s\in S}\ell(s)\leq |x-y|
\label{eq:segments}
\end{equation}
and
\begin{equation}
\cH^{1}(P)\lec |x-y|-\sigma.
\label{eq:paths}
\end{equation}

\label{thm:thefinallemma}
\end{lemma}

In the proof that follows, we set several constants.
Let us mention our order of choosing of the constants we use below so there is no ambiguity.
Let $\phi$ be a small angle that we will fix later. Then we will set $\alpha$ in terms of $\phi$, and $C$ in terms of $\phi$ and $\alpha$. We pick $M$ to be large in terms of $\phi$, $M'$ in terms of $M$, $\delta$ in terms of $M$ and $M'$, and $\ve$ small depending on $M$ and $M'$. 

\subsection{Proof of main lemma}
\begin{proof}
We describe a process of constructing $\gamma$, i.e. obtaining $P$ and $S$ as in the statement of the lemma. We do so inductively.  In particular, we will have a sequence of paths $\gamma_{j}$ such that each path is a union of paths in $\tilde{\Gamma}$ and line segments, and the consecutive $\gamma_{j+1}$ is constructed by replacing each of the segments with another path according to some schema.


Let $L$ be the infinite line through $x$ and $y$, $\Pi$ the projection onto this line, and $\phi>0$ some small angle to be chosen later, and $M'<M$ a large number to be chosen later (in fact, $M'$ will be picked proportionally to $M$). The schema for replacing a segment $[x,y]$ with a new path is organized into four cases:

{
\begin{enumerate}
\item $(x,y)\cap \Gamma\neq\emptyset$
\item $(x,y)\cap \Gamma=\emptyset$
	\begin{enumerate}[(a)]
	\item $\beta(M'B)<\ve$
	\item $\beta(M'B)\geq \ve$. For $z\in[x,y]$, let 
\begin{equation}
z\in (x,y) \;\; \sothat \;\; r_{z}=\sup\{r:B(z,r)\cap \Gamma=\emptyset\}\mbox{ is maximum.}
\end{equation}
		\begin{enumerate}[(i)]
			\item $\beta(B(z,M'r_{z})<\ve$
			\item $\beta(B(z,M'r_{z}))\geq \ve$.
		\end{enumerate}
	\end{enumerate}
\end{enumerate}

We treat  each  case  below.
}

\subsubsection{ Case 1: $(x,y)\cap \Gamma\neq\emptyset$}


Decompose $[x,y]\backslash \Gamma$ into subintervals $\{(a_{i},b_{i})\}$ and let $\gamma$ be the path with segments $[a_{i},b_{i}]$.\\

From now on, we assume $(x,y)\cap \Gamma=\emptyset$. Let $B\in\cB $ be the smallest ball such that $2B\ni x,y$.

\subsubsection{ Case 2.a: $\beta(M'B)<\ve$}

Suppose first that $x,y\in \Delta_{kJ}$ are adjacent. We will construct a sequence of paths $\gamma_{N}$ by adjusting or adding edges. When adjusting a $\gamma_{N}$ to get $\gamma_{N+1}$, we may add some edges that will be permanent in the sense that they will be contained in $\gamma_{i}$ for all $i>n$. We will keep track of these edges by placing them in a collection $S$.

Let $x=x_{0},x_{1},...,x_{n}=y$ be the points in $\Delta_{(k+1)J}$ between $x$ and $y$ 
{
(note that such a path must exist by Remark \ref{r:tsp_extras})}. 
Let $\gamma_{1}'$ be the path obtained by connecting these points in order. If $[x_{j},x_{j+1}]$ is an edge such that $Q=Q(x_{j})\in \cQ_{k+1}$ is   \typeone, then since $\ve$ is small enough, there is a path $p_{Q}\subseteq \tilde{\Gamma}$ and segments $e_{Q}^{1}$ and $e_{Q}^{2}$ with $\ell(e_{Q}^{i})\lec M\ve |Q|$ such that $p_{Q}\cup e_{Q}^{1}\cup e_{Q}^{2}$ is a path connecting $x_{j}$ to $x_{j+1}$. Do similarly if $Q(x_{j+1})$ is  \typeone. Add these edges to our set $S$. Doing this on each edge in $\gamma_{1}'$ makes a new path $\gamma_{1}$. 

{
If $Q$ is \typetwo, then a path $p_{Q}$ and edges $e_{Q}^{1}$ and $e_{Q}^{2}$ exist just as above if $\ve>0$ is small by the discussion after the definition of Condition \star. }

Repeat the above process on each edge $[x_{j},x_{j+1}]$ in $\gamma_{1}'$ that remained in $\gamma_{1}$ (i.e. both $Q(x_{j})$ and $Q(x_{j+1})$ are \typezero) to get a path $\gamma_{2}$ and so on to get a sequence of paths $\gamma_{N}$ that converge to a path $\gamma$ with length $\lec |x_{0}-x_{1}|=|x-y|$ by Lemma \ref{thm:betasmall} and Remark \ref{thm:case2remark}. Moreover, $\gamma\backslash\tilde{\Gamma}=\bigcup_{s\in S}s^{\circ}$, where $s^{\circ}$ denotes the relative interior of $s$. 

By construction, each $s\in S$ is associated to some maximal \typeone or \typetwo cube $Q$ with $\ell(s)\lec M\ve |Q|$ (that is, $Q$ is contained in no other \typeone or \typetwo cube with an edge associated to it),  and each such cube has no more than three edges associated to it. By Lemma \ref{thm:bj}, for $\ve$ small enough,
\[\sigma = \sum_{s\in S}\ell(s)\lec M\ve \sum_{Q\in E(x,y)}|Q|\lec M\ve |x-y|.\]

Let $P=\gamma\cap \Gamma$. Then we can pick $\ve$ small enough so that $\sigma<\ve_{1}|x-y|$, where we pick $\ve_{1}<\frac{1}{4}$ below, and hence
\[\cH^{1}(P)\leq \ell(\gamma)\lec |x-y|\lec |x-y|-\sigma.\]

Now suppose that $x,y$ are arbitrary. Again, we will construct a collection $S$ of edges. Choose $k$ 

{
so that $2^{-(k-1)J}<|x-y|\leq 2^{-(k-2)J}$, so if $x_{0},...,x_{n}\in \Delta_{kJ}$ are those points between $x$ and $y$ ordered by their distance from $x$, then because $J\geq 100$,
\begin{equation} |x_{0}-x|+|x_{n}-y|<\frac{1}{4}|x-y|
\label{e:1/4|x-y|}\end{equation}}
Add $[x_{0},x]$ and $[x_{n},y]$ to $S$. Let $\gamma' =[x,x_{0}]\cup\bigcup_{j=0}^{n-1} [x_{j},x_{j+1}]\cup [x_{n},y]$. Then $\ell(\gamma')\lec |x-y|$. If any edge $[x_{j},x_{j+1}]$ has $Q(x_{j})$ or $Q(x_{j+1})$ not \typezero, as before there is a path $\gamma_{j}=p_{Q}\cup e_{Q}^{1}\cup e_{Q}^{2}$ connecting $x_{j}$ to $y_{j}$, add $e_{Q}^{1}$ and $e_{Q}^{2}$ to $S$. Let $S_{j}=\{e_{Q}^{1},e_{Q}^{2}\}$ and $P_{j}=p_{Q}$. For $\ve$ small, these satisfy

\begin{equation}
\cH^{1}(P_{j})\lec |x_{j}-x_{j+1}|-\sigma_{j}, \;\; \sigma_{j}=\sum_{s\in S_{j}}\ell(s)<\ve_{1}|x_{j}-x_{j+1}|.
\label{eq:pj}
\end{equation}
Otherwise, since $x_{j},x_{j+1}\in \Delta_{kJ}$ are adjacent, we may apply the previous construction to these points  (if $\ve$ is sufficiently small) to get a path $\gamma_{j}$ that is the union of a set $P_{j}$ and collection of segments $S_{j}$ satisfying \eqref{eq:pj} as well. Add the segments from each such $S_{j}$ segments to $S$. 

Note there is a universal constant $C_{0}$ such that
\[|x-x_{0}|+|x_{n}-y|+\sum_{j}|x_{j}-x_{j+1}|\leq C_{0}|x-y|.\]
{
This follows from equation \eqref{e:1/4|x-y|} and because, for $\ve>0$ small enough, the angle between $[x_{j},x_{j+1}]$ and $(x,y)$ is no more than $\frac{\pi}{4}$ for each $j$, and so the polygonal path $x_{1},...,x_{n}$ is a Lipschitz graph along $(x,y)$.}

Let $\ve_{1}<\frac{1}{4C_{0}}$, then by \eqref{eq:pj},
\begin{align*}
\sigma = \sum_{s\in S}\ell(s) & =|x-x_{0}|+|x_{n}-y|+\sum_{j}\sum_{s\in S_{j}}\ell(s) \\
& < \frac{1}{4}|x-y|+\ve_{1}\sum|x_{j}-x_{j+1}|\\
& 
<(\frac{1}{4}+\frac{1}{4})|x-y|  =\frac{1}{2}|x-y|.
\end{align*}
Let
\[\gamma=[x,x_{0}]\cup [x_{n},y]\cup \bigcup_{j} \gamma_{j}\]
and $P=\bigcup P_{j}$. Then 
\[\cH^{1}(P)\lec \sum_{j}|x_{j}-x_{j+1}|
\lec |x-y|\lec |x-y|-\sigma\]
and $\gamma$ thus satisfies the conditions of our lemma and settles the case of when $\beta(M'B)<\ve$.\\

\subsubsection{Case 2.b: $\beta(M'B)>\ve$. Preliminaries}
From now on, we assume $\beta(M'B)>\ve$. \\

%
%
Pick 

\begin{equation}
z\in (x,y) \;\; \sothat \;\; r_{z}=\sup\{r:B(z,r)\cap \Gamma=\emptyset\}\mbox{ is maximum.}
\end{equation}

\noindent Identify $L\ni x,y$ with $\bR$ so that $x>y$ and let $\Pi:\bR^d\to\bR$ be the orthogonal projection onto $L$.
Define
\[H_{z}^{-}:=\{v:y<\Pi(v)<z\}, \;\; H_{z}^{+}:=\{v:z<\Pi(v)<x\}.\]

\noindent That is, $\cH_{z}^{+}$ is the area trapped in between two parallel half planes, once centered at $x$ and another at $z$, that are perpendicular to $L$. 
Similarly for  $\cH_{z}^{-}$. For $j=1,2$, define $V_{j}(z)=C_{\frac{1}{\sin j\phi},L}(z)^{c}$, that is, a cone centered at $z$ with axis $L$. 

\begin{lemma}
For sufficiently small $\phi>0$, if 
\[x'\in H_{z}^{+}\cap V_{1}(z)\]
and

\[y'\in H_{z}^{-}\cap V_{1}(z)\]
are connected by a chord-arc path $p$, then $\gamma=[x,x']\cup p\cup [y,y']$ satisfies \eqref{eq:segments} and \eqref{eq:paths}.
\label{thm:geometriclemma}
\end{lemma}

\begin{proof}

Clearly \eqref{eq:segments} is satisfied. Choose $\phi>0$ small enough so that $\sin \phi<\frac{\cos\phi}{2}$. Let $\phi'<\phi$ be the angle that $[x',z]$ makes with $L$ and $\theta$ the angle $[x,x']$ makes with $L$. Then by the law of sines,
\begin{align*}
|x-x'|(1-\cos\theta) & 
=\frac{\sin^{2}\theta}{1+\cos\theta}|x-x'|
=\frac{\sin\theta\sin\phi'}{1+\cos\theta}|x'-z|
 \leq \sin\phi'|x'-z|\\
& <\frac{\cos\phi'}{2}|x'-z|
\end{align*}

and since $|x-x'|\cos\theta+|x'-z|\cos\phi'=|x-z|$,
\begin{align*}
|x-z|-|x-x'| & 
=(\cos\theta-1)|x-x'|+\cos\phi'|x'-z|
 \geq \frac{\cos\phi'}{2}|x'-z|\\
& =\frac{1}{2}|x''-z|
\end{align*}
where $x''=\Pi(x')$. Similarly,
\[|y-z|-|y-y'|\geq \frac{1}{2} |y-y''|\]
where $y''=\Pi(y')$. Hence,
\begin{align*}
|x-y|-|x-x'|-|y-y'| & 
=|x-z|+|z-y|-|x-x'|-|y-y'| \\
& \geq \frac{1}{2}(|x-x''|+|y-y''|)
=\frac{1}{2}|x''-y''|.
\end{align*}

\noindent Note that the angle that $[x',y]$ makes with $L$ can be no more than $\phi$, thus
\[\ell(p)\lec |x'-y'| \leq \frac{1}{\cos\phi}|x''-y''|\lec |x-y|-|x-x'|-|y-y'|\]
which proves \eqref{eq:paths}.

\end{proof}

\begin{remark} By this Lemma, if we pick $\phi$ small so that $\sin2\phi<\frac{\cos2\phi}{2}$, it now suffices for us to find $x',y'\in \Gamma$ in each component of $V_{2}(z)$ that are connected by a chord arc path in $\tilde{\Gamma}$, which is what we'll do in the next two cases (see Figure \ref{fig:jistofconeargument}).
\end{remark}

\pic{4in}{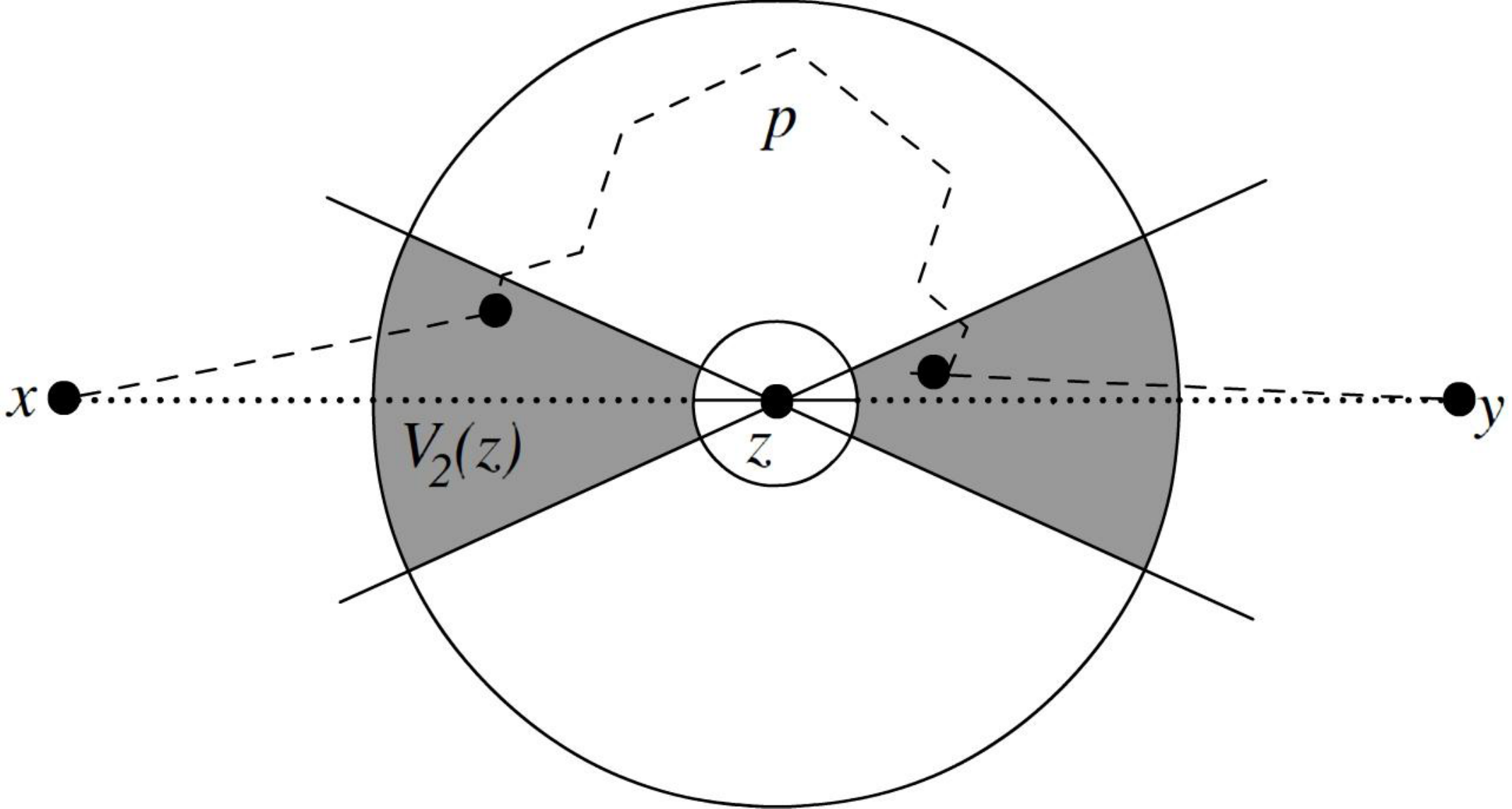}{$\gamma$ as constructed in cases 2.b. (i) and (ii).}{jistofconeargument}

\subsubsection{Case 2.b.i: $\beta(B(z,M'r_{z}))<\ve$} 

We need a proposition that will be used rather frequently in the arguments below:

\begin{proposition}

Identify $L$ with $\bR$ and suppose $a$ and $b$ are points so that $\Pi(a)>\Pi(b)$ and the following hold:

\begin{itemize}
\item[(a)] $\Pi(a)\in (x,y)$
\item[(b)] $a,b\in B(z_{0},\frac{M'r_{0}}{4})$ for some $z_{0}\in L$,
\item[(c)] $|a-b|>cr_{0}$, $c$ some small constant, 
\item[(d)] $\frac{\dist(a,L)-\dist(b,L)|}{|\Pi(a)-\Pi(b)|}\geq \frac{|\dist(a,L)-\dist(b,L)|}{|a-b|}>\frac{10}{M'}$.
\item[(e)] $\beta(B(z_{0},M'r_{0}))<\ve$.
\end{itemize}

Let $z'=\Pi(a)+w+2r_{0}$, where
\[w=\isif{ \frac{2r_{0}(\Pi(a)-\Pi(b))}{\dist(a,L)-\dist(b,L)} & \Pi(a)>\Pi(b) \\
0 & \mbox{ otherwise }}.\]
Then $z'\in B(z_{0},\frac{M'r_{0}}{2})$ and $\dist(z',\Gamma)>\frac{3}{2}r_{0}$ for small $\ve$ (depending on c). Furthermore, $x,y\not\in B(z_{0},M'r_{0})\cap [\Pi(a),\infty)$, so in particular, $z'\in [x,y]$. (See Figure \ref{fig:slope}).
\label{thm:prop}
\end{proposition}

\pic{4.5in}{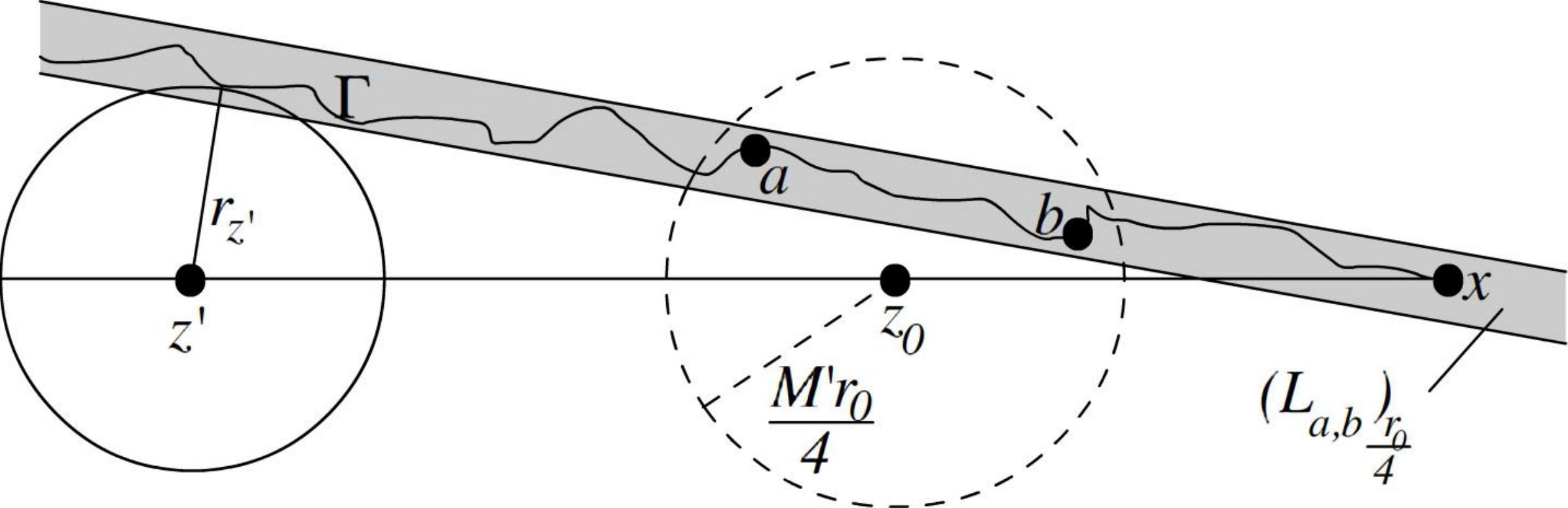}{The heuristic for Proposition \ref{thm:prop} is that if two points $a$ and $b$ satisfy these conditions, then $\Gamma$ is contained in a tube around $L_{a,b}$ that is moving away from $L$, so somewhere along $L$ there must be a point far away from $L_{a,b}$ and hence from $\Gamma$.}{slope}

\begin{proof}

First consider the case when $L$ and $L_{a,b}$ lie in the same two dimensional plane 
{
and $a=\Pi(a)$.} 
Identify this plane with $\bR^{2}$ so that $L=\bR$, $\Pi(a)\geq \Pi(b)$, and $a$ lies in the upper half plane. Suppose $\Pi(a)>\Pi(b)$. Let $f(t)=\dist(t,L_{a,b})$, which, in this case, is a linear function with slope at least $\frac{10}{M'}$. Hence 
\[
f(\Pi(a)+w)  \geq 2r_{0}
\]
and so $\Pi(a)+w+2r_{0}$ is at least $2r_{0}$ from $L_{a,b}$. Thus
\[\dist(L_{a,b},B(\Pi(a)+w+2r_{0},\frac{3}{2}r_{0}))>\frac{1}{4}r_{0}.\]
Moreover, since $\Pi(a)\in B(z_{0},\frac{M'r_{0}}{4})$,
\[|z'-z_{0}|\leq \frac{M'r_{0}}{4}+|w|+2r_{0}\leq r_{0}(\frac{M'}{4}+\frac{2M'}{10}+2)<\frac{M'}{2}r_{0}\]
for large enough $M'$, hence
\begin{equation}
\dist(B(z',\frac{3}{2}r_{0}),L_{a,b}\cup \d B(z_{0},M'r_{0}))\geq (\frac{M'}{2}-\frac{3}{2})r_{0}>\frac{1}{4}r_{0}. \label{e: 1/4r_0}\end{equation}
Therefore, for small enough $\ve>0$, by property (e), $\Gamma\cap B(z_{0},M'r_{0})$ is contained in a tube of radius no more than $2M'r_{0}\ve$, and by property (c), for small enough $\ve>0$ this tube is contained in $(L_{a,b})_{\frac{r_{0}}{4}}$, and thus we have $\dist(z',\Gamma)>\frac{3}{2}r_{0}$. 
The case $\Pi(a)=\Pi(b)$ can be treated similarly. 
{
For the case $\Pi(a)\neq a$,  we just note that $|\Pi(a)-a|<M' r_{0}\ve$, and so by translating the line $L_{a,b}$, the constants above will change by no more than a multiplicative constant if we pick $\ve>0$ small enough.}

Now consider the case when $L$ and $L_{a,b}$ lie in two parallel hyperplanes $H$ and $H_{a,b}$ in $\bR^{d}$. Let $a',b',L_{a,b}'=L_{a',b'}$ be the projections of $a,b,$ and $L_{a,b}$ onto the hyperplane containing $L$ . Since this projection is isometric, $\Pi(a)=\Pi(a'),\Pi(b)=\Pi(b')$, and $\dist(a,L)-\dist(b,L)$ decreases as $H_{a,b}$ moves away from $H$, properties (a) through (d) of the proposition still hold with $a',b'$ instead of $a,b$. Thus the estimate \eqref{e: 1/4r_0} holds with $z'$ as before since it holds with $a'$ and $b'$ in place of $a$ and $b$. The proof that $\dist(z',\Gamma)\geq \frac{3}{2}r_{0}$ is similar to the previous case.

To prove the final statement, notice that if $x\in B(z_{0},M'r_{0})\cap [\Pi(a),\infty)$, then 
\[f(x)\geq f(\Pi(a))=\dist(a,L)\geq \dist(a,L)-\dist(b,L)> |a-b|\frac{10}{M'}>\frac{10cr_{0}}{M'},\]
which means that $x$ can't be contained in $(L_{a,b})_{\frac{r_{0}}{4}}$ if we pick $\ve$ small enough.

\end{proof}

In particular, we have the following useful corollary:

\begin{corollary}
If $r_{0}\geq r_{z}$, there are no points $a,b\in \Gamma$ satisfying (a)-(e).
\end{corollary}

\begin{proof}
This is simply because $r_{z'}>r_{z}$ contradicts the maximality of $r_{z}$. 
\end{proof}

\begin{corollary}
If $z_{0}\in (x,y)$ and $\beta(B(z_{0},M'r_{0}))<\ve$, $r_{0}\geq r_{z}$, then $x,y\not\in B(z_{0},M'r_{0}2^{-4})$.
\label{thm:noxy}
\end{corollary}

\begin{proof}
Suppose $x\in B(z_{0},M'r_{0}2^{-4})$ (the case with $y$ is identical). Let $\xi$ be a point in $B(z_{0},M'r_{z_{0}})$ that projects onto $z_{0}$, so $|z_{0}-\xi|<2M'r_{0}\ve$. Then for $\ve$ small enough,
\[|\xi-x|\geq \dist(\xi,L)\geq \dist(z_{0},\Gamma)=r_{0},\]
\[\frac{\dist(\xi,L)-\dist(x,L)}{|\xi-x|}\geq \frac{r_{0}-0}{|\xi-z_{0}|+|z_{0}-x|}\geq \frac{r_{0}}{2M'r_{0}+M'r_{0}2^{-4}}>\frac{10}{M'}.\]
Hence $\xi$ and $x$ satisfy the conditions of Proposition \ref{thm:prop}, and there exists a point $z'\in [x,y]\cap B(z_{0},\frac{M'r_{z_{0}}}{2})\cap [x,y]$ that is at least $\frac{3}{2}r_{z_{0}}>r_{z}$ from $\Gamma$, contradicting the maximality of $r_{z}$. 
\end{proof}

Let $m$ be such that $B\in \cB_{m}$. 
There is 
\begin{equation} 
\xi_{k}\in B(z,2r_{z})\cap\Delta_{k}, \mbox{ where } 2^{-k}\leq r_{z}<2^{-k+1}.
\end{equation}
 
By Corollary \ref{thm:noxy} applied to $z_{0}=z$, $z_{k}=\Pi(\xi_{k})\in (x,y)$. 
Since $\beta(M'B)>\ve$ and $\xi_{k}\in B$, 
	there is a sequence $\{\xi_{j}\}_{j=m}^{k}$ such that 
	$\xi_{m}$ is the center of $B$, $\xi_{j}\in\Delta_{j}$ and $\xi_{j}\in B(\xi_{j-1},2^{-j})$ and furthermore,
there must be a smallest ball $B(\xi_{j},2^{-j})$, $m\leq j \leq k$ such that $\beta(B(\xi_{j},M'2^{-j}))>\ve$ and $x,y\not\in B(\xi_{j},2^{-j})$. Let $z_{i}$ denote the projection of $\xi_{i}$ onto $L$ for $j\leq i\leq k$, which are contained in $(x,y)$ again by Corollary \ref{thm:noxy}.\\

\begin{corollary}
$|\xi_{i}-z_{i}|<2^{-i+2}$ for $j\leq i\leq k$. 
\end{corollary}

\begin{proof}
Let $\xi_{j}'\in\Gamma$ be such that $|\xi_{j}'-z_{j}|=r_{z_{j}}<r_{z}$. We prove this by induction. If $j=k$, then 
\[|\xi_{k}-z_{k}|\leq \sqrt{|\xi_{k}-z_{k}|^{2}+|z_{k}-z|^{2}}=|\xi_{k}-z|<2r_{z}<2^{-k+2}.\]
Suppose the claim is true down to some $i>j$ but not true for $i-1$, that is,
\begin{equation}
|\xi_{i}-z_{i}|\geq 2^{-(i-1)+2}=2^{-i+3}.
\label{e:xi_i-z_i}
\end{equation}
Then since $\xi_{i-1},\xi_{i}\in \Delta_{i}$,
\[|\xi_{i-1}-\xi_{i}|\geq 2^{-i}.\]

\noindent Furthermore,
\begin{align*}
|\xi_{i-1}-z_{i}| & 
\leq |\xi_{i-1}-\xi_{i}|+|\xi_{i}-z_{i}|
<2^{-i+1}+2^{-k+2}\\
&<2^{-i+1}+2^{-i+2}
=2^{-i+3}.
\end{align*}
{
Here, the second inequality followed by our induction hypothesis that $|\xi_{i}-z_{i}|<2^{-i+2}$.}
Hence $\xi_{i-1},\xi_{i}\in B(z_{i},\frac{M'2^{-i}}{4})$ if we pick $M'>2^{5}$. Furthermore, by the choice of $j$, since $i>j$, we know $\beta(B(z_{i},M'2^{-i}))<\ve$, and by how we chose $\xi_{i}$ and $z_{i}$,
\[|z_{i}-z_{i-1}|\leq |\xi_{i}-\xi_{i-1}|\leq 2^{-i+1},\]
we have
{
\begin{multline*}
\frac{\dist(\xi_{i-1},L)-\dist(\xi_{i},L)}{|\xi_{i-1}-\xi_{i}|} 
=\frac{|\xi_{i-1}-z_{i-1}|-|\xi_{i}-z_{i}|}{|\xi_{i-1}-\xi_{i}|}\\
\geq \frac{2^{-i+3}-2^{-i+2}}{2^{-i+1}}  \geq \frac{2^{-i+1}}{2^{-i+1}}
=1.
\end{multline*}
Here we used \eqref{e:xi_i-z_i} and the induction hypothesis that $|\xi_{i}-z_{i}|<2^{-i+2}$. } Hence $\xi_{i-1},\xi_{i}$ satisfy the the conditions of Proposition \ref{thm:prop}, so there is a $z'\in B(z_{i},\frac{M'2^{-i}}{2})$ with $r_{z'}>\frac{3}{2}2^{-i}>r_{z}$, a contradiction 
{
since $z$ was picked so that $r_{z}$ was maximal}. The second inequality in the corollary follows from the way we picked $k$. 
\end{proof}

\noindent \Claim
\[(B(z_{j},2^{-j+4}\csc\phi)\backslash B(z_{j},2^{-j+3}\csc\phi))\cap \Gamma\subseteq V_{1}(z_{k}).\]
\begin{proof}
We assume $j<k$, since the case of $j=k$ can be proven in a similar manner. Suppose there is a point $\xi\in\Gamma$ in the annulus but outside the cone. Then 

\begin{equation}
\dist(\xi,L)\geq 2^{-j+3}\csc\phi\sin\phi=2^{-j+3} .
\label{eq:outofcone}
\end{equation}

\noindent Then, for $M'>40(2^{4}\csc\phi+3)$, and since our choice of $\phi$ gives $\csc\phi>2$,
\begin{multline*}
|\xi-\xi_{j+1}|\geq |\xi-z_{j+1}|-|z_{j+1}-\xi_{j+1}|
\\
\geq 
2^{-j+3}\csc\phi-2^{-j+3}\geq 2^{-j+4}-2^{-k+3}\geq 2^{-j+3},
\end{multline*}

\[|\xi-z_{j+1}|\leq 2^{-j+4}\csc\phi<\frac{M2^{-j-1}}{4},\]
and
\begin{align*}
\frac{\dist(\xi,L)-\dist(\xi_{j+1},L)}{|\xi-\xi_{j+1}|} & 
\geq \frac{2^{-j+3}-2r_{z}}{|\xi-z_{j}|+|z_{j}-z_{j+1}|+|z_{j+1}-\xi_{j+1}|}\\
& \geq \frac{2^{-j+3}-2^{-k+2}}{2^{-j+4}\csc\phi +2^{-j}+r_{z}}
\geq \frac{2^{-j-2}}{2^{-j}(2^{4}\csc\phi+1+2)}\\
&\geq \frac{10}{M'}.
\end{align*}

\noindent Hence $\xi,\xi_{j+1}\in B(z_{j+1},\frac{M'2^{-j-1}}{4})$ and $\beta(B(z_{j+1},M'2^{-j-1}))<\ve$ by our choice of $j$, thus $\xi$ and $\xi_{j+1}$ satisfy conditions (a)-(e) of the proposition, which is a contradiction, giving us the claim.
\end{proof}

Now, fix $k'$ so that $2^{-k'}\leq \frac{1}{2}|1-e^{i\phi}|<2^{-k'+1}$. Then we may find points $\eta_{i}\in \Delta_{j+k'}$, $i=1,2$, on either component of the cone $V_{2}(z_{j})$ such that 
\[B(\eta_{i},M2^{-j-k'})\supseteq B(\xi_{j},M'2^{-j}),\] 
which is true if we fix $M>M'2^{k'}$. Then

\[ \beta(B(\eta_{i},M2^{-j-k'}))>\frac{M'2^{-j}}{M2^{-j-k'}} \beta (B(\xi_{j},M'2^{-j}))>\delta\ve\]

\noindent for $i=1,2,$ if we pick $\delta<\frac{M'}{M2^{-k'}}$. Hence the balls $B(\eta_{i},2^{-j-k'})$ are \typetwo and satisfy \star (so long as we pick $C>\frac{2\cdot 2^{-j+4}\csc\phi}{2^{-(j+k')}}$) since both points are in $B(z_{j},2^{-j+2})$, hence there are $\eta_{i}' \in B(\eta_{i},M\ve 2^{-j-k'})$ that are contained in the two components of $V_{2}(z_{j})$ that are connected by a chord-arc path $p\subseteq\tilde{\Gamma}$ with length $\lec 2^{-j}$ as desired. (See Figure \ref{fig:case423}.) 
This follows from our choice of $k'$ and that 
\begin{align*}
\dist(\eta_{i}',\Gamma   & \cap (B(z_{j},2^{-j+4}\csc\phi)\backslash B(z_{j},2^{-j+3}\csc\phi)))
 \leq \dist(\eta_{i}',\eta_{i})+2^{-j-k'}\\
& < (M\ve+1) 2^{-j-k'}
 <2\cdot 2^{-j}\frac{1}{2}|1-e^{i\phi}| 
 =2^{-j}|1-e^{i\phi}|
\end{align*}

\noindent hence $\eta_{i}-z_{j}$ makes an angle of no more than $2\phi$ with $L$.

\pic{4.5in}{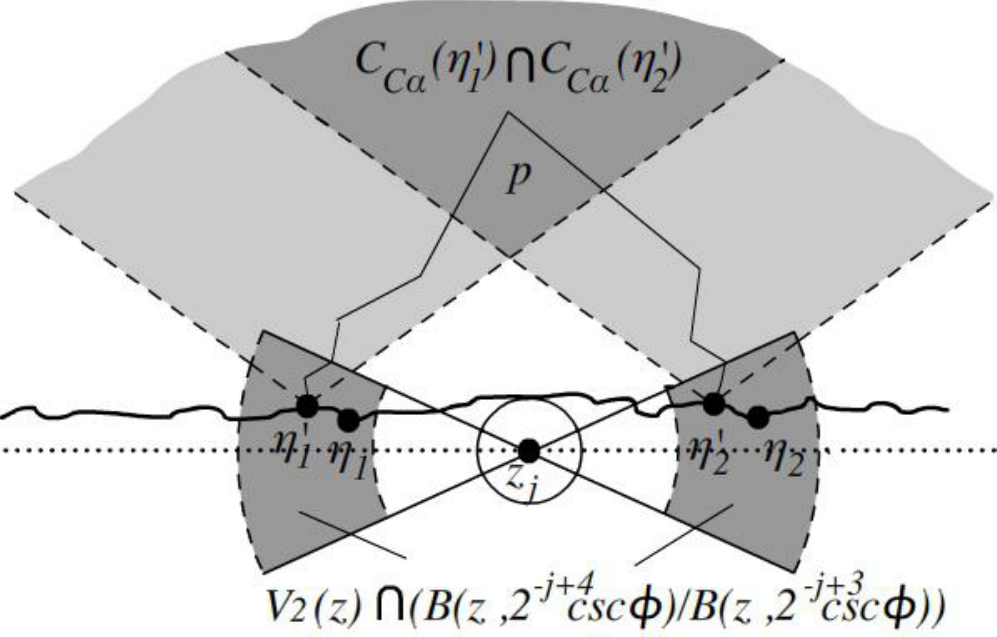}{The $\eta_{j}$ satisfy \star since $\beta(\xi_{j+1},M'2^{-j-1})<\ve$.}{case423}

\subsubsection{Case 2.b.ii: $\beta(B(z,M'r_{z}))\geq \ve$}

\Claim There exist $\xi_{\pm}\in \Gamma\cap V_{1}(z)\cap H_{z}^{\pm} \cap B(z,\alpha r_{z})$ with $z\in C_{\alpha}(\xi_{j})$. 

\begin{proof}
It suffices to find $\xi=\xi_{+}$, since the proof is the same for $\xi_{-}$. Suppose there was no such point. Let $\xi\in \Gamma\cap V_{1}(z)\cap \cH_{z}^{+}\cap B(z,\alpha r_{z})$ be closest to $z_{j}$. Then since $z$ cannot be in $C_{\alpha}(\xi_{j})$,
\[|z-\xi|\geq \alpha \dist(z,\Gamma)=\alpha r_{z},\]
thus $\Gamma\cap V_{1}(z)\cap H_{z}^{+} \cap B(z,\alpha r_{z})=\emptyset$. For large enough $\alpha$, however, we may find another ball centered on $\cH_{z}^{+}\cap [x,y]$ with radius larger than $r_{z}$, which is a contradiction (see Figure \ref{fig:alphacone}).
\end{proof}

\pic{4.5in}{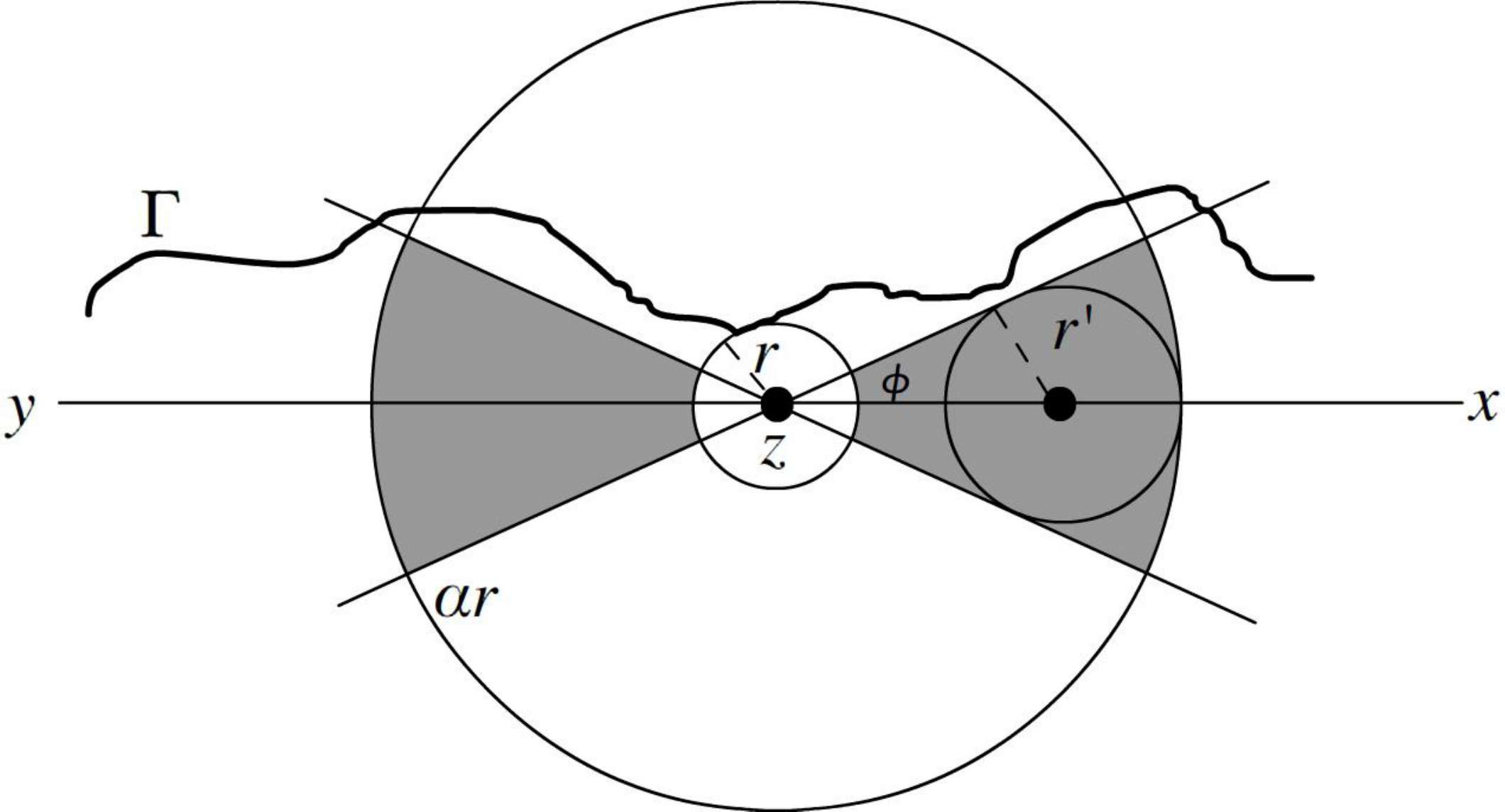}{If $\Gamma\cap V_{1}(z)\cap H_{z}^{+} \cap B(z,\alpha r_{z})=\emptyset$, for large $\alpha$ we may find a ball in the cone with radius larger than $r_{z}$ and disjoint from $\Gamma$.}{alphacone}

Let $\xi_{+}$ be the point in $\Gamma\cap V_{1}(z)\cap \cH^{+}\cap B(z,\alpha r_{z})$ closest to $z$ that satisfies the claim. Identify the two-dimensional plane containing $x,z,$ and $\xi_{+}$ with $\bC$ so that $z=0$, $x>0$, and $\Im \xi_{+}\geq 0$. 

We may pick $C'$ large enough so that $C_{C'\alpha}(\xi_{+})$ contains the path $[\xi_{+},t]\cup[t,z]$, where $t\in\bR$ is such that $\arg(\xi_{+}-t)\geq \frac{\pi}{4}$, and then $C''$ larger so that  $B(z,\frac{r_{z}}{2})\subseteq C_{C''\alpha}(\xi_{+})$.
Note that  we may pick $C''$ independent of $\xi_{+}$ and $z$.

Pick $x'\in B(\xi_{+},2^{-k'})\cap \Delta_{k'}$, where $k'=k'(r_{z})$ is now the integer such that
\[2^{-k'}\leq \frac{1}{2}|r_{z}e^{i\phi}-r_{z}e^{i2\phi}|=\frac{1}{2}r_{z}|1-e^{i\phi}|<2^{-k'+1}.\]

\noindent We may find $y'$ and a $\xi_{-}$ satisfying similar conditions in $\cH_{z}^{-}$, and clearly $|x'-y'|<C''' 2^{-k'}$ for some constant $C'''$.

\pic{4.5in}{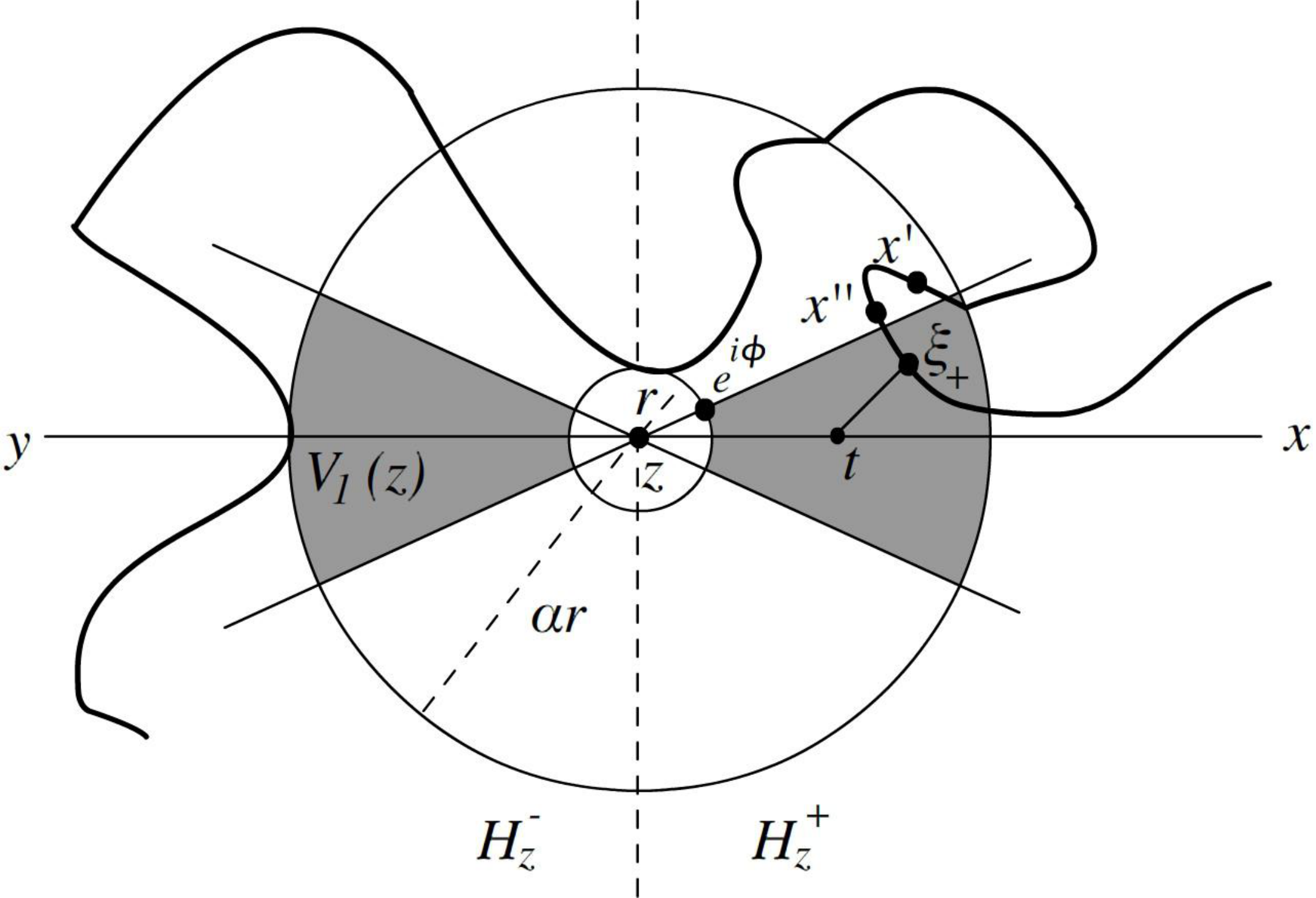}{}{cones}

Now fix $k_{1}$ so that there is $w\in B(z,\frac{r_{z}}{2})\cap \cN_{k'+k_{1}}'$ (recall the relationship between $k'$ and $r$, so we can pick $k_{1}$ independent of these quantities). Then we may connect $\xi_{+}$ to $w$ via the path $[\xi_{+},t]\cup [t,z]\cup [z,w]$, which has length $\lec 2^{-k'}$, and a similar path can be found for $\xi_{-}$. Hence $x'$ and $y'$ satisfy the condition \star if we pick $C>\max\{C'',C'''\}$. Also, 
\[\beta(B(x',M2^{-k'}))\geq \frac{M'r_{z}}{M2^{-k'}}\beta(B(z,M'r_{z}))>\delta\ve\]

\noindent so $B(x',2^{-k'})$ is \typetwo if we pick $\delta<\frac{4M'}{M|1-e^{i\phi}|}$). Hence, we know that there are $x''\in B(x',2^{-k'})$ and $y''\in B(y'',2^{-k'})$ that are connected by a polygonal path of length $\lec 2^{-k'}$, and our choice of $k'$ guarantees that $x'',y''\in V_{2}(z)\cap H_{z}^{+}$, as desired.

\end{proof}

\subsection{Putting it all together}
Using the above lemma, we can fully describe the construction of the curve connecting $x$ and $y$. We run $[x,y]$ through our schema to get a new curve $\gamma_{1}$, which is a union of segments $S_{1}$ and a set $P_{1}$ that satisfy
\[\sigma_{1}=\sum_{s\in S_{1}}\ell(s)<|x-y|\]
and
\[\cH^{1}(P_{1})\lec |x-y|-\sigma_{1}.\]

On each segment $s\in S_{1}$, we replace each of them with new paths $\gamma_{s}$, which are unions of segments $s'\in S_{s}$ and a set $P_{s}$ such that
\[\sigma_{s}=\sum_{s'\in S_{s}}\ell(s')\leq \ell(s)\]
and
\[\cH^{1}(P_{s})\leq C_{1}(\ell(s)-\sigma_{s}).\]

After replacing each of these segments, we form a new path $\gamma_{2}$ connecting $x$ and $y$ that is a union of segments $S_{2}$ and a set $P_{2}=P_{1}\cup \bigcup_{s\in S_{1}}P_{s}$ such that
\begin{align*}
\sigma_{2} & =\sum_{s\in S_{2}}\ell(s)=\sum_{s\in S_{1}}\sum_{s'\in S_{s}}\ell(s') \leq \sum_{s\in S_{1}}\ell(s)<|x-y|
\end{align*}

\noindent and
\begin{align*} \cH^{1}(P_{2})& 
\leq \cH^{1}(P_{1})+\sum_{s\in S_{1}}\cH^{1}(P_{s})
 \leq C_{1}(|x-y|-\sigma_{1})+\sum_{s\in S_{1}}C_{1}(\ell(s)-\sigma_{s}) \\
& = C_{1}|x-y|-C_{1}\sum_{s\in S_{1}}\sigma_{s}=C_{1}\ps{|x-y|-\sum_{s\in S_{2}}\ell(s)}\\
& = C_{1}(|x-y|-\sigma_{2}).
\end{align*}

Inductively, we may construct a sequence of curves $\gamma_{n}$ such that each $\gamma_{n}$ is a union of segments $S_{n}$ and a set $P_{n}$ such that
\[ \sigma_{n}:=\sum_{s\in S_{n}}\ell(s)<|x-y|\]
and
\[\cH^{1}(P_{n})\leq C_{1}(|x-y|-\sigma_{n}).\]

Hence, they converge to a Lipschitz curve $\gamma$ that is contained in $\tilde{\Gamma}$ that satisfies $\ell(\gamma)\leq (C_{1}+1)|x-y|$. This combined with Lemma \ref{thm:wlog} and Lemma \ref{thm:length} finishes the proof of the main theorem.

\bibliographystyle{amsalpha}
\bibliography{reference}

\providecommand{\bysame}{\leavevmode\hbox to3em{\hrulefill}\thinspace}
\providecommand{\MR}{\relax\ifhmode\unskip\space\fi MR }
\providecommand{\MRhref}[2]{%
  \href{http://www.ams.org/mathscinet-getitem?mr=#1}{#2}
}
\providecommand{\href}[2]{#2}
\begin{thebibliography}{GJM92}

\bibitem[Bis10]{Bishop-Tree}
Christopher~J. Bishop, \emph{Tree-like decompositions and conformal maps}, Ann.
  Acad. Sci. Fenn. Math. \textbf{35} (2010), no.~2, 389--404. \MR{2731698
  (2011m:30009)}

\bibitem[BJ94]{BJ}
Christopher~J. Bishop and Peter~W. Jones, \emph{Harmonic measure, {$L^2$}
  estimates and the {S}chwarzian derivative}, J. Anal. Math. \textbf{62}
  (1994), 77--113. \MR{1269200 (95f:30034)}

\bibitem[BJ97]{BJ97}
C.~J. Bishop and P.~W. Jones, \emph{Wiggly sets and limit sets}, Ark. Mat.
  \textbf{35} (1997), no.~2, 201--224. \MR{MR1478778 (99f:30066)}

\bibitem[Chr90]{Christ-T(b)}
Michael Christ, \emph{A {$T(b)$} theorem with remarks on analytic capacity and
  the {C}auchy integral}, Colloq. Math. \textbf{60/61} (1990), no.~2, 601--628.
  \MR{1096400 (92k:42020)}

\bibitem[Dav88]{David-T(b)}
Guy David, \emph{Op\'erateurs d'int\'egrale singuli\`ere sur les surfaces
  r\'eguli\`eres}, Ann. Sci. \'Ecole Norm. Sup. (4) \textbf{21} (1988), no.~2,
  225--258. \MR{956767 (89m:42014)}

\bibitem[DN95]{Das-Narasimhan}
Gautam Das and Giri Narasimhan, \emph{Short cuts in higher dimensional space},
  Proc. of the 7th Canadian Conference on Computational Geometry, 1995,
  pp.~103--108.

\bibitem[GJM92]{GJM}
John~B. Garnett, Peter~W. Jones, and Donald~E. Marshall, \emph{A {L}ipschitz
  decomposition of minimal surfaces}, J. Differential Geom. \textbf{35} (1992),
  no.~3, 659--673. \MR{1163453 (93d:53010)}

\bibitem[GM08]{Harmonic-Measure}
John~B. Garnett and Donald~E. Marshall, \emph{Harmonic measure}, New
  Mathematical Monographs, vol.~2, Cambridge University Press, Cambridge, 2008,
  Reprint of the 2005 original. \MR{2450237 (2009k:31001)}

\bibitem[Jon90]{Jones-TSP}
Peter~W. Jones, \emph{Rectifiable sets and the traveling salesman problem},
  Invent. Math. \textbf{102} (1990), no.~1, 1--15. \MR{1069238 (91i:26016)}

\bibitem[KK92]{KK}
Claire Kenyon and Richard Kenyon, \emph{How to take short cuts}, Discrete
  Comput. Geom. \textbf{8} (1992), no.~3, 251--264, ACM Symposium on
  Computational Geometry (North Conway, NH, 1991). \MR{1174357 (93h:65184)}

\bibitem[Ler03]{Le}
G.~Lerman, \emph{Quantifying curvelike structures of measures by using {$L\sb
  2$} {J}ones quantities}, Comm. Pure Appl. Math. \textbf{56} (2003), no.~9,
  1294--1365. \MR{2004c:42035}

\bibitem[Mat95]{Mattila}
Pertti Mattila, \emph{Geometry of sets and measures in {E}uclidean spaces},
  Cambridge Studies in Advanced Mathematics, vol.~44, Cambridge University
  Press, Cambridge, 1995, Fractals and rectifiability. \MR{1333890 (96h:28006)}

\bibitem[Mit04]{mitchell-spn-04}
Joseph S.~B. Mitchell, \emph{Shortest paths and networks}, Handbook of Discrete
  and Computational Geometry (2nd Edition) (Jacob~E. Goodman and Joseph
  O'Rourke, eds.), Chapman \& Hall/CRC, Boca Raton, FL, 2004, pp.~607--641.

\bibitem[NS07]{spanner-book}
Giri Narasimhan and Michiel Smid, \emph{Geometric spanner networks}, Cambridge
  University Press, Cambridge, 2007. \MR{2289615 (2009b:68002)}

\bibitem[Oki92]{O-TSP}
Kate Okikiolu, \emph{Characterization of subsets of rectifiable curves in
  {${\bf R}^n$}}, J. London Math. Soc. (2) \textbf{46} (1992), no.~2, 336--348.
  \MR{1182488 (93m:28008)}

\bibitem[Sch07]{Schul-TSP}
Raanan Schul, \emph{Subsets of rectifiable curves in {H}ilbert space---the
  analyst's {TSP}}, J. Anal. Math. \textbf{103} (2007), 331--375. \MR{2373273
  (2008m:49205)}

\end{thebibliography}

\end{document}